\documentclass[11pt]{amsart}
\usepackage[all]{xy}
\usepackage{amsmath, amssymb, amscd, amsthm, mathtools, amsfonts, mathrsfs, tikz-cd, bbm, enumitem}
\usepackage[mathscr]{euscript}
\usepackage[utf8]{inputenc}
\usepackage[english]{babel}
\usepackage{mathbbol}
\usepackage{xcolor}
\usepackage[title]{appendix}
\usepackage{tikz}
\usepackage{tikz-cd}
\tikzcdset{scale cd/.style={every label/.append style={scale=#1},
    cells={nodes={scale=#1}}}}
\usepackage{bold-extra}

\usepackage{graphicx}
\usepackage{color}
\usepackage{epsfig}
\usetikzlibrary{positioning}
\usepackage{url}
\usepackage[colorlinks=true,citecolor=blue,urlcolor=blue,linkcolor=blue]{hyperref}
\usepackage{setspace}
\usepackage{geometry}
\usepackage{graphics}

\usepackage{adjustbox}
\usepackage{color,soul}

\usetikzlibrary{matrix}
\usepackage{fancyhdr}
\usepackage{pdfpages}

\usepackage{verbatim}
\newcommand {\DR} {\mathbf{DR}}
\newcommand {\Hom} {\mathsf{Hom}}
\newcommand {\Tor} {\mathsf{Tor}}

\newcommand {\TFH} {\mathsf{TFH}}

\newcommand {\End} {\mathcal E\mathsf{nd}}
\newcommand {\map} {\mathsf{Map}}

\newcommand {\Type} {\textsf{Type}}

\newcommand  {\dg}     {\mathbf{dg}}
\newcommand  {\rank}     {\textsf{rank}}

\newcommand  {\medg}     {\epsilon-\mathbf{dg}^{gr}}

\newcommand  {\cdga}     {\mathrm{cdga}}

\newcommand  {\dAff}     {\mathbf{dAff}}

\newcommand  {\Mod}   {\mathrm{Mod}}
\newcommand  {\QCoh}   {\mathrm{QCoh}}
\newcommand  {\Bun}   {\mathrm{Bun}}

\newcommand  {\dSt}   {\mathbf{dSt}}

\DeclareMathOperator{\dAffit}{dAff}

\DeclareMathOperator{\Sym}{Sym}
\DeclareMathOperator{\Tot}{Tot}
\DeclareMathOperator{\Spec}{Spec}
\DeclareMathOperator{\codim}{codim}
\DeclareMathOperator{\Aut}{Aut}
\DeclareMathOperator{\Coh}{\mathsf{Coh}}
\DeclareMathOperator{\tr}{tr}
\theoremstyle{plain}

\numberwithin{equation}{section}
\newcommand{\spec}{{\rm Spec}\,}
\theoremstyle{definition}
\newtheorem{Def}{Definition}[section]

\newtheorem{Rem}[Def]{Remark}
\newtheorem{Rem*}[]{Remark}

\theoremstyle{plain}
\newtheorem{Prop}[Def]{Proposition}
\newtheorem{Thm}[Def]{Theorem}
\newtheorem{Lem}[Def]{Lemma}
\newtheorem{Cor}[Def]{Corollary}

\begin{document}
\title{Moduli stacks of Higgs bundles on stable curves}
\author{Oren Ben-Bassat, Sourav Das, Tony Pantev}
\date{\today}

\setlength{\abovedisplayskip}{16pt}
\setlength{\belowdisplayskip}{10pt}
\maketitle

\section*{\textbf{Abstract}}
We construct a flat degeneration of the derived moduli stack of Higgs bundles on smooth projective curves, using Jun Li’s stack of bounded expanded degenerations. The degeneration carries a natural relative zero shifted logarithmic symplectic form over the base, which extends the classical Hitchin symplectic structure from the generic fiber. The induced Hitchin map is shown to be complete (in the valuative sense) and flat, without requiring coprimeness of rank and degree. Finally, we extend the construction globally: the derived moduli stack of Higgs bundles over the universal semi-stable curve of genus greater than or equal to two carries a relative zero shifted log-symplectic form over the moduli stack of stable curves.

\tableofcontents

\section{\textbf{Introduction}}
It is well known that moduli stacks of vector bundles on nodal curves are not complete. Often, their completions involve adding torsion-free sheaves as points of the boundary. Because we want to use the mapping space techniques of shifted symplectic geometry, we adopt the alternative approach of completing the moduli by adding boundary points, which are bundles on a bubbling node. This approach goes back to the classical works of David Gieseker \cite{12} and Jun Li \cite{Li}. We will implement the approach in our setting by utilizing Jun Li's stack of expanded degenerations. This allows us to use only vector bundles instead of coherent torsion-free sheaves and apply techniques of shifted symplectic geometry \cite{PTVV}, obtaining a relative symplectic structure on a complete moduli stack. We extend the Hitchin map to Higgs bundles on the bubbled curves parametrized by the stack of expanded degenerations and then use the completeness of this moduli stack to prove that our extended Hitchin map is complete.
The derived moduli stack of Higgs bundles on a smooth curve $C$ has virtual dimension \[-\chi_{\mathrm{Dol}}(C, \End(E,\phi)) = n^2 \deg(c_1(T_C)) = 2n^2(g-1)\]
 where $g$ is the genus of $C$, while the derived moduli stack of vector bundles (or any other Lagrangian in the derived moduli stack of the Higgs bundle) has a virtual dimension \[-\chi(Y, \End(E))=\int_{C} \text{ch}(\End(E))\text{td}(C)=n^2(g-1).\] These moduli stacks are homotopically locally of finite presentation (see \cite[\S 2.2.6.3]{TV} and \cite{32a}), and so exhibit the "hidden smoothness" envisioned in the derived deformation theory. Given this, it is natural to study other geometric properties of these moduli. This article investigates the relative $0$-shifted symplectic structures in the total space of a degeneration of $C$ to a nodal curve.

After writing this article, we realized that there are many relationships with the articles \cite{BDD}, \cite{BDDP} and significant overlap with a forthcoming article \cite{DF} of Ron Donagi and Andres Fernandez Herrero. In their work, they give a construction of a good moduli space for the semistable locus and semistable reduction relative to the Hitchin fibration (so the moduli space is proper over the Hitchin base in families). This uses the ``infinite dimensional GIT" picture developed with Dan Halpern-Leistner. They also showed the flatness of the Hitchin morphism and that things are syntomic (as classical stacks), including the symplectic leaves. They are pursuing the story with punctures, thinking about the log Poisson reduction picture for the relative log cotangent stack of framed Gieseker vector bundles.

To start with, we fix a family of projective curves $\mathcal X$ on the spectrum of a discrete valuation ring $S$ so that the generic fiber is smooth, the closed fiber is an irreducible nodal curve with a single node, and the total space $\mathcal X$ is smooth. We also fix a rank $n$ and degree $d$ for the vector bundles in our moduli problem. In Section \S2, we recall the construction of the stack $\mathfrak M$ of bounded expanded degenerations (bounded by the integer $n$) of the family of curves $\mathcal X/S$ following \cite{Li} and \cite{Z}. One of the main results in this subsection is the lemma \ref{et1}, which we use to prove the following proposition.

\

\noindent
{\bfseries Proposition~\ref{rellog101}} \ {\em
The morphism $\mathfrak M\longrightarrow S$ is a log-smooth map, and the relative log-cotangent complex $\mathbb L^{log}_{\mathfrak M/S}=0$. }

\

In subsection \S2.3, we discuss shifted log-symplectic structures on quasi-smooth derived Artin stacks (Definition \ref{qs}) equipped with a locally free log structure. In subsection \S2.4, we recall the definition of relative shifted symplectic forms for a quasi-smooth morphism of derived Artin stacks. We define relative shifted log-symplectic forms for certain logarithmic morphisms of derived Artin stacks equipped with locally-free log structures.
In Section \S3, we construct a logarithmic version of the relative Dolbeault moduli stack for the universal expanded family of curves $\mathcal \mathcal X_{\mathfrak M}\rightarrow \mathfrak M$. We show that the relative logarithmic Dolbeault moduli stack has a relative $0$-shifted log-symplectic form over $S$. Moreover, we show that the relative log-symplectic form is an extension of Hitchin's symplectic form on the generic fibre of the moduli stack over $S$. This was proved for moduli schemes in \cite{D}. The main results of this section are the following.

\

\noindent
{\bfseries Theorem~\ref{main}} \ {\em $\mathcal{X}_{ Dol}$ is $\mathcal{O}$-compact and $\mathcal O$-oriented over $S$. Hence,
$\map_{\mathcal{S}}(\mathcal X_{ Dol}, BGL_n\times \mathcal S)$ has a $0$-shifted relative symplectic structure over $S$.}

\

\noindent
{\bfseries Theorem~\ref{main112}} \ {\em The derived Artin stack $\map_{_{\mathfrak M}}(\mathcal X_{\mathfrak M, Dol}, BGL_n\times \mathfrak M)$ has a natural relative $0$-shifted log-symplectic structure over $S$.}

\

\noindent
In Section \S4, we define the Hitchin map on the classical Artin stack of Gieseker-Higgs bundles $M^{cl}_{Gie}$ (see \ref{Gies21}). We prove that the Hitchin map is complete.

\

\noindent
{\bfseries Theorem~\ref{complete11}} \ {\em The morphism $h|_{M^{cl}_{Gie}} \colon M^{cl}_{Gie} \longrightarrow B$ is complete.}

\

This result was proved in \cite{2a} for the Hitchin map on the moduli scheme where rank and degree are co-prime. We prove it here for the moduli stack, and the argument does not require us to assume that the rank and degree are coprime.

In section \S5, we study the reduced global nilpotent cone of $M^{cl}_{Gie}$, which is the reduction of the scheme theoretic fibre over the point $0\in B$. We prove that every irreducible component of the reduced nilpotent cone has an open subset (denoted by $\mathcal Nilp^{sm, gen}$), which is an isotropic substack of $M$ (the derived stack of Higgs bundles) with respect to its log-symplectic form. We use this to compute the dimension of the reduced nilpotent cone and to show that the Hitchin map is flat. The main theorem of this section is as follows.

\

\noindent
{\bfseries Theorem~\ref{isotropic112}}
{\em \begin{enumerate}
\item The Hitchin map $h: M^{cl}_{Gie}\longrightarrow B$ is surjective.
\item The sub-stack $\mathcal Nilp^{sm,gen}$ is relatively isotropic over $S$.
\item The Hitchin map $h: M^{cl}_{Gie}\longrightarrow B$ is flat.
\end{enumerate}
}

\

In Section \S6, we construct the Gieseker-like derived moduli stack of Higgs bundles $\mathcal M^{Dol}_g$ over the moduli stack of stable curves of genus $g\geq 2$. We prove the following theorem.

\

\noindent
{\bfseries Theorem~\ref{main1567}} \ {\em
There is a $0$-shifted relative log-symplectic form on $\mathcal M^{Dol}_g$ (relative to the moduli stack of stable curves
$\overline{\mathcal M_g}$).
}

\

In the Appendix, we construct the relative classical Artin stack of Gieseker-Higgs bundles and study its local properties. The main results of the appendix are Proposition \ref{new1} and Theorem \ref{dim111}. In the first lemma, we prove that the stack of Gieseker vector bundles is an almost very good stack. We use this in Lemma \ref{dim111} to show that the classical stack of Gieseker-Higgs bundles is an irreducible, local complete intersection.

\

\noindent
{\bfseries Proposition \ref{new1}} \ {\em
The closed fibre $N^{cl}_{Gie,0}$ is an irreducible, equidimensional, almost very good stack (\cite[Definition 2.1.2]{32e}) with normal crossing singularities.
}

\

\noindent
{\bfseries Theorem \ref{dim111}} \ {\em
The stack $M^{cl}_{Gie,0}$ is an irreducible local complete intersection of pure dimension $2\dim N_{Gie,0} + 1$.
}

\

\subsection{\textbf{Notations and Conventions}}
\begin{enumerate}\label{ConNot}
\item[$\bullet$] $\mathbb k$ be an algebraically closed field of characteristic zero. Every geometric object lives over $\Spec(\mathbb k)$. However, the $\mathbb k$ is usually suppressed from the notation.

\item[$\bullet$] $S:=\{\eta, o\}$ denotes the spectrum of a complete discrete valuation ring over $\mathbb k$, where $\eta$ denotes the generic point, and $o= \Spec(\mathbb k)$ denotes the closed point.

\item[$\bullet$] $\mathcal X \to S$ denotes a flat family of curves whose generic fibre is smooth projective, and the closed fibre is a nodal curve with a single node. We denote the nodal curve by $X_0$ and the node by $x$. We denote its normalisation by $q:\widetilde X_0\longrightarrow X_0$ and the two pre-images of the node $x$ by $\{x^+,x^{-}\}$. 

\item[$\bullet$] $\dg$ is the category of dg-modules over $\mathbb k$ (i.e. of
complexes of $\mathbb k$-modules). By convention, the differential of an
object in $\dg$ \emph{increases} degrees.  So an object is a cochain complex of the form \[\cdots \to E_{-1} \to E_0 \to E_1 \to \cdots. \]

\item[$\bullet$] $\cdga$ is the category of commutative dg-algebras over $\mathbb k$,
and $\cdga^{\leq 0}$ its full subcategory of non-positively graded
commutative dg-algebras.

\item[$\bullet$] $dg$, $\cdga$ (respectively $\cdga^{\leq 0}$) are
  endowed with their natural model structures for which equivalences
  are quasi-isomorphisms, and fibrations are epimorphisms
  (respectively, epimorphisms in strictly negative degrees).

\item[$\bullet$] ${\dAffit}:=(\cdga^{\leq 0})^{op}$ is the category of derived affine $\mathbb k$-schemes.

\item[$\bullet$] The $\infty$-categories associated with the model categories $dg, \cdga^{\leq0}, \dAffit$ are denoted by $\textbf{dg}, \textbf{cdga}^{\leq0}, \textbf{dAff} $.

\item[$\bullet$] The $\infty$-category of simplicial sets is denoted by $\mathbb{S}$. It is also called the $\infty$-category of spaces, and space will be used to mean a simplicial set.

\item[$\bullet$] The $\infty$-category of derived stacks over $\mathbb k$, for the \'{e}tale topology, is denoted by $\textbf{dSt}$. If $X$ is a derived stack, the $\infty$-category of derived stacks over $X$ is denoted by $\textbf{dSt}_X$. The classical truncation of a derived stack $X$ is denoted by $t_0 X$.

\item[$\bullet$] Our running convention is that unless explicitly specified otherwise a \linebreak \emph{\bfseries (derived)  Artin stack} will always mean a \emph{\bfseries (derived) geometric 
$1$-stack}. In particular, all the moduli stacks we consider are geometric $1$-stacks.

\item[$\bullet$] For a family of semistable curves $\mathfrak X$ over a scheme $T$, $\mathfrak X_{Dol}$ denotes the relative logarithmic Dolbeault shape of $\mathfrak X/T$.
\begin{center}
\begin{tikzcd}
\mathfrak X_{Dol}\arrow["q"]{r}\arrow[bend right=20,swap, "p"]{rr}  & \mathfrak X \arrow["r"]{r} & T
\end{tikzcd}
\end{center}

\end{enumerate}
\section{\textbf{Acknowledgements}}
T.P. (University of Pennsylvania) was partially supported by NSF FRG grant DMS-
2244978, NSF/BSF grant DMS-2200914, NSF grant DMS-1901876, and Simons Collaboration grant number 347070. O.B. (University of Haifa) and S.D. would like to acknowledge the Binational Science Foundation and NSF/BSF grant DMS-2200914 also known as BSF Grant 2021717 for supporting S.D. as a postdoc. S.D. would like to thank Professor Alek Vainshtein of the University of Haifa, Israel, for the financial support from his Israel Science Foundation grant number 876/20 during a postdoc position at the University of Haifa. S.D. would like to thank Professor Vikraman Balaji for the financial support from his SERB Core Research Grant CRG/2022/000051-G during a postdoc position at Chennai Mathematical Institute. S.D. also thanks IISER Tirupati for their support during a part of this work. 
\section{\textbf{Preliminaries}}

\subsection{\textbf{Space of bounded expanded degenerations}}
In this subsection, we recall a construction by Jun Li \cite{Li} called the stack of expanded degenerations. We will use this stack to construct our degeneration. The construction of Jun Li starts with a one-parameter family of varieties (of any dimension) such that the total space of the family is smooth, the generic fibre is smooth, and the special fibre is a normal crossing divisor in the total space of the family. For our purposes, we only need to consider the expanded degenerations of such a family for curves. We will make use of the standard fact that for any nodal curve, one can always construct a smoothing over the spectrum of a discrete valuation ring, which also has a smooth total space. The precise setting we consider can be explained below.

\subsubsection{\textbf{Degeneration of curves.}}\label{DegeCourbes} Start with a flat family of projective curves $\mathcal X\longrightarrow S$, such that 
\begin{enumerate}
\item generic fiber $\mathcal X_{\eta}$ is a smooth curve of genus $g\geq 2$, 
\item the closed fibre is a nodal curve $X_0$ with a single node $x \in X_{0}$, and
\item the total space $\mathcal X$ is regular over $\spec \mathbb k$. 
\end{enumerate}

\begin{Rem}
Given any nodal curve, such a family of one-parameter degeneration of curves can always be constructed where the given nodal curve is placed as the closed fiber \cite[Theorem B.2 and Corollary B.3, Appendix B]{M1}. On the other hand, given any smooth proper family of curves $\mathcal X_K$ over $\Spec K$, where $K$ is the function field of the discrete valuation ring $S$, one can always extend (possibly after taking a finite cover of $S$) it to a family of stable curves with nodal singularities(possibly with several nodes) \cite[Theorem 109.24.3.]{34}. Also note that Jun Li's construction of the stack of expanded degenerations also works for the case when the closed fiber has more than one node (see \cite[page 516]{Li}).   
\end{Rem}

Let us denote the relative dualizing sheaf by $\omega_{\mathcal X/S}$. \textbf{From now on, we set $S:=\Spec~ \mathbb k[t]_{(t)}$, which represents an open subset of the affine line $\mathbb A^1_{\mathbb k}$.} 

\smallskip

\begin{Def}\label{gis}(\textbf{Gieseker curve/Expanded degeneration/Modification}) 
Let $X_{0}$ be a nodal curve with a single node $x \in X_{0}$, and let $x^{\pm}$ label the two preimages of
$x$ in normalization $\widetilde{X}_{0}$ of $X_{0}$. Let $r$ be a positive integer.
\begin{enumerate}
\item A \emph{\bfseries chain of $r$ projective lines} is a scheme $R[r]$ of the form
$\cup_{i=1}^{^r} R[r]_{i}$ such that 
\begin{enumerate}
\item $R[r]_{i}\cong \mathbb{P}^{^1}$, 
\item for any $i<j$, $R[r]_{i}\cap R[r]_{j}$ consists of a single point $p_j$ if $j=i+1$ and is empty otherwise. 
\end{enumerate}

We call $r$ the length of the chain $R[r]$. Let us choose and fix two smooth points $p_1$ and $p_{r+1}$ on $R[r]_1$ and $R[r]_{r}$, respectively. 

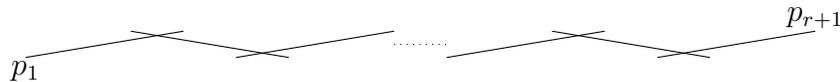
\vspace{1cm}
\begin{tikzpicture}[overlay, xshift= 1cm,scale=0.70]
\draw node[yshift=-1ex]{$p_1$}(0,0) -- (3,.5) node[yshift=1ex]{}; \draw (2,.5) -- (5,0)node[yshift=-1ex]{}; \draw (4,0) -- (7,.5); \draw[dotted] (7,.25) -- (8,.25);\draw (8,0) -- node[yshift=1.5ex]{}(11,.5); \draw (10,.5) -- (13,0)node[yshift=-1ex]{}; \draw (12,0) -- (15,.5) node[yshift=1ex]{$p_{r+1}$};
\end{tikzpicture}
\vspace{1cm}

\item A \emph{\bfseries Gieseker curve} $X_r$ is the categorical quotient of the disjoint union of the curves $\tilde{X}_0$ and $R[r]$ obtained by identifying $x^+$ with $p_1$ and $x^-$ with $p_{r+1}$. We will also refer to these curves as \emph{\bfseries expanded degenerations} or \emph{\bfseries bubblings} of the nodal curve $X_0$.
\end{enumerate}
\end{Def}

\begin{Def}(\textbf{Family of Gieseker curves / expanded degenerations / modifications:\cite[Definition 3.4]{2a}}) \label{gis120} For every $S$-scheme $T$, a \emph{\bfseries modification of $\mathcal{X}$} relative to $T$  is a commutative diagram
\begin{equation}
\begin{tikzcd}
\mathfrak X_T\arrow{dr}[swap]{p_T}\arrow{rr}{\pi_T}&& \mathcal X_T:=\mathcal X\times_S T\arrow{dl}\\
& T
\end{tikzcd}
\end{equation}
such that
\begin{enumerate}
\item $p_T : \mathfrak X_T \longrightarrow T$ is flat;
\item the horizontal morphism is finitely presented and  is an isomorphism over smooth  fibers $(\mathcal X_T )_t$ of $\mathcal{X}_{T} \to T$;
\item over each closed point $t\in T$ that maps to  $\eta_0\in S$, we have $(\mathfrak X_T)_t$ is isomorphic to a Gieseker curve $X_r$ for some integer $r$ and the horizontal morphism restricts to the morphism $X_{r} \to X_{0}$ that contracts the $\mathbb P^1$s in  $X_r$ to the node $x \in X_{0}$.
\end{enumerate}
\end{Def}

\smallskip

\begin{Def}(\textbf{Morphisms of families of expanded degenerations:})
Let $T \to S$ be a scheme mapping to $S$, and let $\mathfrak X_T$ and ${\mathfrak{X}'}_{T}$ be two $T$-relative  modifications of $\mathcal{X} \to S$.

We call $\mathfrak X_T$ and $\mathfrak X'_T$ \emph{\bfseries isomorphic} if there exists an isomorphism $\sigma_T: \mathfrak X_T\longrightarrow \mathfrak X'_T$ such that the following diagram commutes
\begin{equation}
\begin{tikzcd}
\mathfrak X_T\arrow{dr}[swap]{\pi_T} \arrow{rr}{\sigma_T} && \mathfrak X'_T\arrow{dl}{\pi'_T}\\
& \mathcal X_T
\end{tikzcd}
\end{equation}
\end{Def}

\smallskip

\begin{Rem}
The definition of a modification or Gieseker curve is stated slightly differently from the definition given in Gieseker's original paper (see \cite[(i) and (iii) of the first paragraph of \S 4]{12}. The definition given in this paper is the same as the Gieseker's definition (see \cite[(ii) of Definition 7; Lemma. on page 200]{25}). For the deformation theory of Gieseker curves (i.e., the description of such families of modifications over Artin rings), we refer to \cite[Proposition 4.5. ]{12}, and \cite[Local Theory, Page 197]{25}. 
\end{Rem}

\subsubsection{\textbf{Space of expanded degenerations}}
As before, we will choose an uniformizer in $S$ at the closed point and if necessary replace it by $S:=\Spec~ \mathbb k[t]_{(t)}$, which represents an open neibourhood of the origin of the affine line $\mathbb A^1_{\mathbb k}$. For any positive integer $n$, we set
\begin{equation}
S[n]:=S\times_{\mathbb A^1} \mathbb A^{n+1}
\end{equation}
where the map $\mathbb A^{n+1}\longrightarrow \mathbb A^1$ is given by $(t_1, \dots, t_{n+1})\mapsto t_1\cdots t_{n+1}$. 

Consider the group
\begin{equation}
G[n]:=\mathbb G_m^{n}
\end{equation}
acting on $\mathbb A^{n+1}$ by 
\begin{equation} \label{eq:S[n]action}
(\sigma_1, \ldots, \sigma_n)\cdot (t_1, \ldots, t_{n+1}):=(\sigma_1\cdot t_1,\ldots, \sigma_{i-1}^{-1}\cdot\sigma_{i}\cdot t_i, \ldots, \sigma_{n}^{-1}\cdot t_{n+1}).
\end{equation}
The map $\mathbb{A}^{n+1} \to \mathbb{A}^{1}$, $(t_{1},\ldots,t_{n+1}) \mapsto t_{1}t_{2}\cdots t_{n+1}$ intertwines this action with the trivial action on $\mathbb{A}^{1}$ and hence 
\eqref{eq:S[n]action} induces an action of $G[n]$ on $S[n]$.

In \cite[Section 1.1]{Li}, Jun Li constructed a $G[n]$-equivariant family of expanded degenerations $W[n]$ over $S[n]$. It is constructed by several birational transformations of the family $\mathcal X\times_S S[n]$. The curves that occur in the family are all possible expanded degenerations of the family $\mathcal X/S$ for which the length of rational chains is bounded by $n$.

\smallskip

\textbf{Notation: To be consistent with our notations, from now on, we will denote Jun Li's family of expanded degenerations by $\mathfrak X[n]$ over $S[n]$.}

\subsubsection{\textbf{Stack of bounded expanded degenerations}} 
The significance of the family $\mathfrak X[n] \to S[n]$ is that it can be used as an atlas for the universal family of expanded degenerations with bubbling of length $\leq n$.

Indeed, recall from \cite[Definition 1.9, Proposition 1.10]{Li} the following
\begin{Def}\label{Exp1123} The \emph{\bfseries stack of expanded degenerations} of $\mathcal{X}/S$ is the stack
$$
\mathfrak{M}: {Sch}/S\longrightarrow {Groupoids} 
$$ 
given by the assignment
\begin{equation}
T
\mapsto \left\{
\text{
\begin{minipage}[c]{4.2in}
Pairs $(\mathfrak X_T, \pi)$ where $T$ is a scheme, $\mathfrak X_T\to T$ is a family of projective curves, and $\pi: \mathfrak X_T\to \mathcal X$ is a morphism over $S$ such that there exists an \'{e}tale cover
$\widetilde{T}\to T$ and a map $\widetilde{T}\to S[n]$ so that  $\mathfrak X_T\times_T \widetilde{T}\cong \mathfrak X[n]\times_{S[n]} \widetilde{T}$
\end{minipage}}
\right\}
\end{equation}

\smallskip

\noindent
Two such families $\mathfrak X_T$ and $\mathfrak X'_{T}$ are isomorphic if there is a $T$-isomorphism \linebreak $f: \mathfrak X_T \longrightarrow \mathfrak X'_{T}$ compatible with tautological projections $\mathfrak X_T \longrightarrow \mathcal X$ and $\mathfrak X'_T \longrightarrow \mathcal X$.
\end{Def}

\begin{Rem}
By construction, the stack of expanded degenerations has the following properties 
\begin{enumerate}
\item $\mathfrak M$ is a smooth Artin stack of finite type.
\item The projection map $\mathfrak M\longrightarrow S$ is generically an isomorphism.
\item The closed fibre of $
\mathfrak{M} \to S$ is a normal crossing divisor in $\mathfrak{M}$.  
\end{enumerate}
\end{Rem}

The following properties are obtained from the construction / Definition of the expanded degeneration stack as the quotient of $\mathbb A^{n+1}$ by a smooth equivalence relation (see \cite[page 16, Def. 2.22.]{Z}). 

\smallskip

Here we use the following. 
\begin{Def}
Let $\mathcal Y$ be a smooth Artin stack and $\mathcal D$ be a closed sub-stack of co-dimension one. We say that $\mathcal D$ is a \emph{\bfseries normal crossing divisor} in $\mathcal{Y}$ if for any smooth morphism $f: T\rightarrow \mathcal Y$ from a smooth scheme $T$, the pull-back of the divisor $\mathcal D\times_{\mathcal Y} T$ is a normal crossing divisor in $T$.
\end{Def}

\subsubsection{\textbf{An alternative construction of the stack \texorpdfstring{$\mathfrak M$}{M}}}In \cite[Definition 2.22]{Z}, Zijun Zhou gave an alternative construction of the stack $\mathfrak M$ which will be useful for our purposes. Here we briefly recall his construction.

\smallskip

First, let us introduce some useful terminology and notation. For any integer $n$, we set $[n]:=\{1, \dots, n+1\}$. It will be useful to have short names for the natural maps between the spaces $S[k]$ for various values of $k$.

\begin{Def}
Given any subset $I\subseteq [n]$ of cardinality, say $k+1$, we define
\begin{itemize}
\item The \emph{\bfseries Standard Embedding} corresponding to  $I$ is the embedding
\begin{equation}
\gamma_I: S[k]\hookrightarrow S[n]
\end{equation}
given by $\gamma_I(t_1,\dots, t_{k+1})=(z_1, \dots, z_{n+1})$, where

$$
(z_1, \dots, z_{n+1})=\begin{cases}
			t_i, & \text{when} ~~i\in I\\
            1, & \text{when} ~~ i\in [n]\setminus I
		 \end{cases}
$$

Notice that $S[k]:=S\times_{\mathbb A^1} \mathbb A^{k+1}$ and $(t_1, \dots, t_{k+1})$ denote a point of $S[k]$. 

\smallskip

\item The \emph{\bfseries Standard Open Embedding} corresponding to $I$   is the open embedding
\begin{equation}
\tau_I: S[k]\times \mathbb{G}_m^{n-k}\hookrightarrow S[n]
\end{equation}
where $S[k]\times \mathbb{G}_m^{n-k}$ denotes the open subset swept  by $S[k]$ under  the free action of the subgroup $\mathbb{G}_m^{n-k}\subset G[n]$ whose $j$-th component is $1$ if $j\notin I$.
\end{itemize}
\end{Def}

\smallskip

\begin{Def}\label{Def.2.6}
\begin{description}
\item[(i)] Given two subsets $I$ and $I'$ of $[n]$ both of order $k+1$, define the equivalence relation $R_{I,I'}$ on $S[n]$  by setting 
\begin{equation}
R_{I,I'} = S[k]\times \mathbb{G}_m^{n-k}\rightrightarrows S[n]
\end{equation}
where the two maps are $\tau_I$ and $\tau_{I'}$.
\item[(ii)]  For every $k \leq 0$ define a discrete equivalence relation $R_{k} \rightrightarrows S[n]$ as
\begin{equation}
R_{k}:=\coprod_{|I|=|I'|=k+1, I, I'\subseteq [n]} R_{I, I'}  \ \rightrightarrows \ S[n].
\end{equation}
\item[(iii)] Define the amalgamated discrete equivalence relation $R_{dis}$ on $S[n]$ by setting
\begin{equation}
R_{dis}:=\coprod_{k\leq n} R_{k} \ \rightrightarrows \ S[n]
\end{equation}
\item[(iv)] Finally, we define the total smooth groupoid 
$$R_{tot}:=\mathbb{G}_m^n\times R_{dis} \rightrightarrows S[n]$$ 
which is generated by $R_{dis}$ and the action of $\mathbb{G}_m^n$ on $S[n]$.
\end{description}
\end{Def}

\smallskip

\begin{Def}\label{Def.2.7} Let 
$\mathfrak M_n:=\left[S[n]/R_{tot}\right]$
be the quotient stack of $S[n]$ by the smooth groupoid $R_{tot} \rightrightarrows S[n]$.
\end{Def}

\smallskip

\noindent
From \cite[Remark 2.23]{Z}, it follows that the stack $\mathfrak M_n$ is the stack of expanded degenerations of the family $\mathcal X/S$ bounded by the integer $n$. 

\smallskip

\noindent
{\textbf{Notation:}} \ {From here onwards, we will only work with $\mathfrak M_n$ and we drop the subscript $"n"$ from the notation and denote it simply by $\mathfrak M$.}

\smallskip

\begin{Lem}\label{et1}
The morphism between quotient stacks 
\begin{equation}
[S[n]/G[n]]\longrightarrow \mathfrak M
\end{equation}
is \'{e}tale.
\end{Lem}
\begin{proof}
Let us first show that the relation $R_{dis}$ descends to an equivalence relation on the quotient stack $[S[n]/G[n]]$).
Recall that $R_{dis}:=\underset{k\leq n}\coprod R_{k}$ and \[R_{k}:=\underset{|I|=|I'|=k+1, I, I\subseteq [n]}\coprod R_{I, I'},\] where 
the equivalence relations $R_{I,I'}$ are given by 
\begin{equation}
S[k]\times \mathbb{G}_m^{n-k}\rightrightarrows S[n]
\end{equation}
Recall that the map $R_I: S[k]\times \mathbb{G}_m^{n-k}\longrightarrow S[n]$ is given by the standard open embedding. 

For simplicity, assume that $S=\mathbb A^1$. Then $S[n]=\mathbb A^{n+1}$. Notice that $G[n]:=\mathbb G_m^{n}$. But in fact, there is an action of a larger group $\mathbb G_m^{n+1}$ on $\mathbb A^{n+1}$ given by coordinate multiplication. Moreover, the group $\mathbb G_m^n$ can be identified as a subgroup of $\mathbb G_m^{n+1}$ consisting of elements of the form $(t_1, \dots, t_{n+1})$ such that $\prod^{n+1}_{i=1} t_i=1$. Let us consider the action of the larger group $G[n+1]$. 

Let us first define the natural action of $G[n+1]$ on $R_I:=\mathbb A^{k+1}\times \mathbb{G}_m^{n-k}$. Notice that the cardinality of the set $I$ is equal to $k+1$. It is an ordered subset of $[n+1]$. Therefore, the complement $J:=[n+1]\setminus I$ is also an ordered set of cardinality $n-k$, where the order is the induced order of $[n+1]$. Now we see that $\mathbb G_m^{n+1}\cong \mathbb G_m^{I}\times \mathbb G_m^J$. Finally, we define the action of $\mathbb G_m^{n+1}$ on $R_I:=\mathbb A^I\times \mathbb{G}_m^J$ by the obvious action of $\mathbb G_m^{I}\times \mathbb G_m^J$. It is easy to check that the map $R_I: \mathbb A^{k+1}\times (\mathbb k^*)^{n-k}\longrightarrow \mathbb A^{n+1}$ is $\mathbb G_m^{n+1}$-equivariant. Hence, it is also equivariant under the action of the smaller subgroup $\mathbb G_m^{n}$. Therefore, the equivalence relations $R_{I, I'}$ descend to the quotient stack $[S[n]/G[n]]$.

Now we will show that the relation $\underline{R}_{dis}$ is an \'{e}tale equivalence relation. First of all, note that the maps $R_I: S[k]\times \mathbb{G}_m^{n-k}\longrightarrow S[n]$ are open immersions. Therefore, $R_{k}:=\coprod_{|I|=|I'|=k+1, I, I\subseteq [n]} R_{I, I'}\rightrightarrows \mathbb A^{k+1}$ is an \'{e}tale equivalence relation because both the projections are Zariski open immersions. Therefore,  $R_{dis}$ on $S[n]$ and $[R_{dis}/G[n]]$ on $[S[n]/G[n]]$ define an equivalence relation \'{e}tale equivalence relation.
\end{proof}

\smallskip

\subsection{\textbf{Construction of the family of curves}}\label{Fam} In this subsection, we will briefly recall from \cite{Li, Z} the fact that there is a family of expanded degeneration (whose rational chain has a length bounded by $n$). Let us choose a neighborhood \'{e}tale $U_p$ of node $p\in \mathcal X$ such that the nodal curve \linebreak $U_p\times_{\mathcal X} X_0$ is a reducible nodal curve with two smooth connected components intersecting transversely at node $p$. Let $V = \mathcal{X} - \{p\}$. Then $\{U_{p},V\}$ is an 
\'{e}tale covering of $\mathcal{X}$. We have the following diagram (which is both cartesian and co-cartesian/push-out):

\begin{equation}
\begin{tikzcd}
U_p\times_{\mathcal X} V\arrow{d} \arrow{r} & V \arrow{d}\\
U_p\arrow{r} & \mathcal X
\end{tikzcd}
\end{equation}
Notice that $U_p\longrightarrow S$ is a simple normal crossing degeneration that we can feed into the construction of \cite{Li, Z} to produce a family of expanded degenerations. We denote this family by $U_p[n]\longrightarrow S[n]$ and also write $V[n]:=V\times_{S} S[n]$. Now we construct the family of expanded degenerations, denoted by $\mathcal{X}[n]$ over $S[n]$, by the following push-out
\begin{equation}
\begin{tikzcd}
U_p[n]\times_{(\mathcal X\times_S S[n])} V[n]\arrow{d} \arrow{r} &  V[n] \arrow{d}\\
U_p[n]\arrow{r} & \mathcal X[n]
\end{tikzcd}
\end{equation}
The pushout exists as an algebraic space according to Rydh's pushout theorem \cite[Theorem C]{R}.
The family of curves $\mathcal X[n]\longrightarrow S[n]$ is the desired family of expanded degenerations of the original family $\mathcal X\longrightarrow S$. 

We refer to \cite[Proposition 2.13]{Z} for the key properties of the expanded degeneration family. In particular, by properties (3) and (4) of \cite[Proposition 2.13]{Z} the family $U_p[n] \to S[n]$ is equivariant under the action of the smooth groupoid $R_{tot}\rightrightarrows S[n]$. By the universality of pushouts, this implies that the family of expanded degenerations $\mathcal{X}[n] \to S[n]$ is also equivariant and so $\mathcal{X}[n]$ descends to the stack of expanded degenerations $\mathfrak{M} = \mathfrak{M}_{n} = \left[S[n]/R_{tot}\right]$. We denote this family of expanded degenerations by $\mathfrak{X}_\mathfrak{M} \to \mathfrak{M}$.

\begin{Prop}
    The family $\mathfrak{X}_\mathfrak{M} \to \mathfrak{M}$ has the following universal property: Given any family of expanded degenerations (bounded by the integer $n$) $\mathfrak X_T\to T$ (over a scheme $T$) of the family of curves $\mathcal X/S$, there exists a unique morphism $\rho: T\to \mathfrak M_n$ such that $\rho^*\mathcal X\cong \mathfrak X_T$. 
\end{Prop}

\begin{proof}
    The proof follows from \cite[Lemma 1.8. , Definition 1.9., Proposition 1.10.]{Li} and \cite[2.3 Standard families of expanded degenerations and pairs – gluing, and 2.5 Stacks of expanded degenerations and pairs]{Z}. The Lemma \cite[Lemma 1.8.]{Li} proves the universal property of the stack and the expanded family for the case when $\mathcal X/S$ is a "simple-degeneration", which means that the closed fiber of the family $\mathcal X/S$ is a union of two smooth curves that intersect transversely at a point. Then following the arguments of \cite[2.3 Standard families of expanded degenerations and pairs – gluing, and 2.5 Stacks of expanded degenerations and pairs]{Z} the universal property follows in our case, i.e., when the closed fibre is an irreducible nodal curve. Notice that even though Zijun explained the construction for the simple-degeneration case (for higher dimensional DM-stacks), his arguments work for our case as well, which is the case when the closed fiber of $\mathcal X/S$ is an irreducible nodal curve. We merely explained the construction here because for an irreducible nodal curve the construction involves a much simpler choice of etale covers.
\end{proof}

\smallskip

\subsection{\textbf{Log structures on derived Artin stacks}}
In this subsection, we discuss locally-free log structures on a quasi-smooth derived Artin stack. We refer to \cite{26,27} and \cite{Pr} for the prerequisite material on log-structure/ locally free log-structure on Artin/ derived Artin stacks. We now recall a few necessary definitions.

Let $\mathfrak L^0$ denote the algebraic stack that classifies fine log structures, and let $\mathfrak L^1$ denote the algebraic stack that classifies morphisms of fine log schemes  \cite{27}. We denote by $\mathfrak{L}^0_{\mathsf{lf}}$ and $\mathfrak L^1_\mathsf{lf}$ the classifying substacks  (known to be smooth by \cite[Proposition 1.7]{Pr}) of locally free log structures and morphisms between locally free log structures, respectively. Given a logarithmic scheme $(S, P, \alpha)$, the stack $\mathfrak L^0_{S,\mathsf{lf}}:=\mathfrak L^1_\mathsf{lf}\times_{\mathfrak L^0_{\mathsf{lf}}} S$ is the classifying stack of logarithmic morphisms from a scheme with locally free log structure to $(S, P,\alpha)$. With this notation, we now have the following natural definition \cite{Pr}:

\smallskip

\begin{Def} \label{def:logstrstacks}
A \emph{\bfseries locally free log structure on a derived algebraic stack $\mathfrak X$} is a morphism $\mathfrak X\longrightarrow \mathfrak L^0_{\mathsf{lf}}$ of derived 
stacks. A \emph{\bfseries morphism of derived log stacks $\mathfrak X\longrightarrow \mathfrak Y$ equipped with locally-free log structures} is a commutative diagram of derived stacks
\begin{equation}
\begin{tikzcd}
\mathfrak X\arrow{r}\arrow{d} & \mathfrak L^1_{\mathsf{lf}}\arrow{d}\\
\mathfrak Y\arrow{r} & \mathfrak L^0_{\mathsf{lf}}
\end{tikzcd} 
\end{equation}
\end{Def}

\smallskip

\subsubsection{\textbf{Relative logarithmic cotangent complex}}

Given a log morphism $f: \mathfrak X\longrightarrow \mathfrak Y$ between derived stacks equipped with locally free log structures, \emph{\bfseries the relative log-cotangent complex of the morphism $f$} is defined by 
\begin{equation}
{\mathbb L}^{\log}_{\mathfrak X/ \mathfrak Y}:=\mathbb L_{\mathfrak X/ \mathfrak{L}^0_{\mathfrak{Y},\mathsf{lf} }}
\end{equation}
Here, as before,  $\mathfrak{L}^0_{\mathfrak{Y},\mathsf{lf}} := \mathfrak{L}^{1}_{\mathsf{lf}}\times_{\mathfrak{L}^{0}_{\mathsf{lf}}} \mathfrak{Y}$ and  $\mathbb L_{\mathfrak X/ \mathfrak L^0_{\mathfrak{Y},\mathsf{lf}}}$ denote the relative cotangent complex for the morphism of stacks $\mathfrak X\longrightarrow  \mathfrak L^0_{\mathfrak{Y},\mathsf{lf}}$.

With this definition in mind, we return to the discussion on the stack of bounded expanded degenerations $\mathfrak M\rightarrow S$ (\ref{Exp1123}, \ref{Def.2.7}) for a given one parameter degeneration $\mathcal X\rightarrow S$ of curves.

\smallskip

\begin{Prop}\label{rellog101}
The morphism $\mathfrak M\longrightarrow S$ is a log-smooth map. Moreover, the relative log-cotangent complex $\mathbb L^{\log}_{\mathfrak M/S}=0$. 
\end{Prop}
\begin{proof}
From Lemma \ref{et1}, it follows that the map $[S[n]/ G[n]]\longrightarrow \mathfrak M$ is \'{e}tale. Notice that the varieties $S[n]$ and $S$ are log smooth varieties with the log structure coming from the divisor given by the pre-image of the closed point of $S$ under the map $S[n]\longrightarrow S$. Since the map $S[n]\longrightarrow S$ is an \'{e}tale base-change of the map $\mathbb A^{n+1}\longrightarrow \mathbb A^1$ given by $(t_1, \dots, t_{n+1})\mapsto t_1\cdots t_{n+1}$, the map is a log-smooth morphism. Therefore, it follows that the map $[S[n]/ G[n]]\longrightarrow S$ is also a log-smooth morphism of log-smooth Artin stacks. 

The explicit description of the relative log-cotangent complex is the following. 
The pull-back (to $S[n]$) of the log-cotangent complex of the stack $[S[n]/G[n]]$ is the perfect 
complex on $S[n]$ concentrated in degrees $0$ and $1$ which is given by 
\begin{equation}\label{moment}
\left[\Omega_{S[n]}^{1}(\log ~\partial S[n])\longrightarrow \mathcal O_{S[n]}^{\oplus n}\right]
\end{equation} 
where $\Omega_{S[n]}^{1}(\log ~\partial S[n])$ is the rank $(n+1)$ locally free sheaf of logarithmic one forms on $S[n]$ with poles along the strict normal crossings divisor $\partial S[n] = S[n]\times_{S} \{o\}$.  Since the action of $G[n]$ on $S[n]$ respects $\partial S[n]$, we obtain a natural $G[n]$-equivariant structure $\Omega_{S[n]}^{1}(\log ~\partial S[n])$  such that the natural log-symplectic structure is $G[n]$-equivariant. The morphism \eqref{moment} is given by the (equivariant)  moment map for the action of $G[n]$ on the log-cotangent bundle of $S[n]$. In particular \eqref{moment} is a $G[n]$-perfect complex. We recall that the action of $G[n]$ on $S[n]$ is given by 
\begin{equation}\label{act123}
(t_1, \dots, t_{n+1})\cdot(x_1, \dots, x_{n+1})=(t_1x_1, t_2 x_2, \dots, t_{n+1}x_{n+1})
\end{equation}
where $t_1\cdots t_{n+1}=1$, and the map $S[n]\longrightarrow S$ is given by

\begin{equation}\label{act}
(t_1, \dots, t_{n+1})\mapsto \prod^{n+1}_{i=1} t_i
\end{equation}

It is easy to see that the map \eqref{moment} 
\begin{equation}
\left[\Omega_{S[n]}^{1}(\log ~\partial S[n])\longrightarrow \mathcal O_{S[n]}^{\oplus n}\right]
\end{equation} 
of vector bundles is surjective: consider the action of $(\mathbb G_m)^n$ on $\mathbb A^{n+1}$ given by \eqref{act123}. Let $p:=(x_1, \cdots, x_{n+1})\in \mathbb A^{n+1}$ be any point. Then consider the orbit map at the point $p$. 

\begin{align}
    Or_p: (\mathbb G_m)^n\longrightarrow \mathbb A^{n+1}\\
    (t_1, \cdots, t_{n+1})\mapsto (t_1\cdot x_1, \cdots, t_{n+1}\cdot x_{n+1})
\end{align}

The induced map on the tangent spaces is $d(Or)_p: T_{(1,\cdots, 1)}(\mathbb G_m)^n=\mathbb k^n\longrightarrow T_p \mathbb A^{n+1}=\mathbb k^{n+1}$ is given by the multiplication by the following diagonal matrix

\smallskip

\[
\begin{bmatrix}
    x_1       & 0 & 0 & \dots & 0 \\
    0       & x_2 & 0 & \dots & 0 \\
    \hdotsfor{5} \\
    0       & 0 & 0 & \dots & x_{n+1}
\end{bmatrix}
\]

\smallskip

This implies that the image of $[\sum^{n+1}_{i=1} \lambda_i \partial_{t_i}]\in T_{(1,\cdots, 1)}(\mathbb G_m)^n=\mathbb k^n$ (notice that $\sum^{n+1}_{i=1} \lambda_i=0$) is $\sum^{n+1}_{i=1} \lambda_i\cdot x_i \cdot \partial_{x_i}=\sum^{n+1}_{i=1} \lambda_i\cdot \partial_{\log~x_i}$. Therefore, we get a map 

\begin{align}
    d(Or)_p: T_{(1,\cdots, 1)}(\mathbb G_m)^n=\mathbb k^n\longrightarrow (T \mathbb A^{n+1}(-\log~\partial A^{n+1}))_p=\mathbb k^{n+1}\\
    (\sum^{n+1}_{i=1} \lambda_i\cdot \partial_{t_i})\mapsto \sum^{n+1}_{i=1} \lambda_i\cdot \partial_{\log~x_i}
\end{align}

The map above is clearly injective. Therefore, the dual of this map 

\begin{align}
    (d(Or)_p)^{\vee}: \Omega^1_p \mathbb A^{n+1}(\log~\partial A^{n+1})\longrightarrow \Omega^1_{(1,\cdots, 1)}(\mathbb G_m)^n
\end{align}

is surjective. This is the explicit description of the moment map \ref{moment}. It is easy to see that the composite map
\begin{equation}
\Omega_{S}^{1}(\log o)\hookrightarrow \Omega_{S[n]}^{1}(\log ~~\partial S[n])\longrightarrow \mathcal O_{S[n]}^{\oplus n}
\end{equation}
is $0$, and that the morphism $\Omega_{S[n]/S}^{1}(\log ~~\partial S[n])\longrightarrow \mathcal O_{S[n]}^{\oplus n}$ is an isomorphism.\\

Notice that the perfect complex
\begin{equation}
\big[\Omega_{S[n]/S}^{1}(\log ~~\partial S[n])\longrightarrow \mathcal O_{S[n]}^{\oplus n}\big]
\end{equation}
is precisely the relative log cotangent complex of the morphism $\big[ S[n]/G[n]\big]\longrightarrow S$ pulled back to $S[n]$. But since the morphism $\Omega_{S[n]/S}^{1}(\log ~~\partial S[n])\longrightarrow \mathcal O_{S[n]}^{\oplus n}$ is an isomorphism, the relative log-cotangent complex of the morphism $\big[ S[n]/G[n]\big]\longrightarrow S$ is equivalent to $0$.
\end{proof}

\smallskip

\subsection{\textbf{Relative shifted log-symplectic forms}}
In this subsection, we will recall the definition of relative shifted symplectic forms for a locally of finite presentation morphism of derived Artin stacks. We will then define relative shifted log-symplectic forms for certain logarithmic morphisms of derived Artin stacks equipped with locally-free log structures. 

Let $\pi: M\longrightarrow S$ be a morphism of derived Artin stacks with $S$ a scheme. Suppose we also assume that the stack $M$ and the scheme $S$ are equipped with locally free log structures such that the map $\pi$ is a morphism of log stacks. These data are equivalent to a map of derived  stacks $\pi_{log}: M\longrightarrow \mathfrak L^0_{S,\mathsf{lf}}$, where $\mathfrak L^0_{S,\mathsf{lf}}$ is the classifying log stack for locally free log schemes  mapping into $S$. We further assume that the map $\pi_{log}$ is of locally of finite presentation. 

\smallskip

\noindent
It is well-known that a map $R \longrightarrow S$ of commutative $\mathbb{k}$-algebras is finitely presented if and only if $\Hom_{R/calg}(S,-)$ commutes with filtered colimits. 
Above we used the homotopical notion of being locally finitely presented, defined as follows:

\smallskip

\begin{Def} \cite[Definition 2.16]{PV} \ We say that a map $A\longrightarrow B$ in $\cdga^{\leq 0}$ is (homotopically) \emph{\bfseries finitely presented} if $\map_{A/\cdga^{\leq 0}}(B, -) $ commutes with (homotopy) filtered colimits.
\end{Def}

\begin{Rem} \label{Rem0910}
A map $A \longrightarrow B$ in $\cdga^{\leq 0}$ is finitely presented (\cite[Definition 2.16]{PV} and \cite[Theorem 7.4.3.18]{20}) if and only if
\begin{enumerate}
\item $H^0(A)\longrightarrow H^0(B)$ is classically finitely presented, and
\item $\mathbb L_{B/A}$ is a perfect $B$-dg module (i.e., a dualizable object in $(dgmod(B), \otimes)$). 
\end{enumerate}
\end{Rem}

\smallskip

\noindent
We will also need an appropriate notion of quasi-smoothness, which we recall next:

\smallskip

\begin{Def} 
\begin{enumerate}[label=\upshape(\alph*),ref=\thesubsubsection (\alph*)]
\item\label{qs} A map of derived Artin stacks $f: \mathcal X\longrightarrow \mathcal Y$ is \emph{\bfseries quasi-smooth}  if it is locally of finite presentation and the relative cotangent complex $\mathbb L_f$ concentrated in degrees $\geq -1$.
\item\label{LQS}
Let $\pi: M\rightarrow S$ be a logarithmic morphism from a derived Artin stack $M$ to a scheme $S$ both equipped with locally free log structures. We say that the map $\pi$ is \emph{\bfseries log quasi-smooth} if the corresponding classifying  morphism $\pi_{log} : M \to \mathfrak{L}^{0}_{S,\mathsf{lf}}$ is quasi-smooth.
\end{enumerate}
\end{Def}

\smallskip

\begin{Def}{\cite[Definition 2.12]{Pr}}
The \emph{\bfseries relative log cotangent complex} of the log morphism of derived stacks $\pi : M \to S$ is the complex ${\mathbb {L}}^{\log}_{M/S}:=\mathbb L_{\pi_{\log}}$. Here $\mathbb L_{\pi_{\log}}$ denotes the relative cotangent complex for the classifying  morphism $\pi_{\log}$.
\end{Def}

\smallskip

\begin{Rem} When $\pi : M \to S$ is a log quasi-smooth morphism from a derived Artin stack to a scheme $X$, then the quasi-smoothness  of the  map $\pi_{\log}: M\longrightarrow \mathfrak L^0_{S, \mathsf{lf}}$ is equivalent to the fact that ${\mathbb L}^{\log}_{M/S}$ is a perfect complex over $M$ of amplitude $[-1, 1]$.
\end{Rem}

Now we recall the definitions of relative shifted symplectic forms for a finitely presented map of derived Artin stacks $M\longrightarrow N$ from \cite{PTVV}, \cite[Definition 1.4.1]{CPTVV}. The finitely presented condition implies that the relative cotangent and relative tangent complexes are perfect (Remark \ref{Rem0910}). 

\smallskip

Let $A \in \cdga^{\leq 0}$ and let $M:=\Spec A\longrightarrow N$ be a finitely presented morphism to a derived Artin stack $N$. The  \emph{\bfseries relative de Rham algebra} for $M \to N$ is the cdga $DR(M/N) \in \cdga^{\leq 0}_{A}$ given by 
\begin{equation}
DR(M/N):=\Sym_A(\mathbb L_{M/N}[1])) = \oplus^{\infty}_{p=0} \wedge^p \mathbb L_{M/N}[p],
\end{equation}
where $\mathbb L_{M/N}$ is the relative cotangent complex, and $DR(M/N)$ is equipped with the  (cohomological) differential $d$ induced from the differential in $\mathbb L_{M/N}$. 

\smallskip

\noindent
We also have the de Rham differential $d_{DR}: \wedge^p \mathbb L_{M/N}\longrightarrow \wedge^{p+1} \mathbb L_{M/N}$. The de Rham differential defines a mixed structure on $DR(M/N)$ with $\epsilon:=d_{DR}$. This makes $DR(M/N)$ into a mixed cdga (we include the definition below). In addition to the internal cohomological grading $DR(M/N)$ has a second (form degree) ``weight grading" defined as $DR(M/N)(p):= \wedge^p \mathbb \mathbb L_{M/N}[p]$. The weight grading and mixed structure on $DR(M/N)$ are also compatible with the multiplicative structure and make $DR(M/N)$ into a graded mixed cdga over $\mathbb k$. The degree and weight of the internal differential $d$ are $1$ and $0$, respectively. The degree and weight of $\epsilon$ are $-1$ and $1$, respectively. For the definitions of mixed complexes, graded mixed complexes and mixed cdga's we refer the reader to \cite[Section 1.1]{PTVV}. 

\smallskip

Let $dAff_{/N}$ be the comma category of derived affine schemes mapping to $N$.
The assignment $(\Spec A\longrightarrow N)\mapsto DR(\Spec A/N)$ defines a functor 
\begin{equation*}
dAff_{/N}^{\text{op}} \longrightarrow \epsilon-dg^{gr}
\end{equation*}
from $dAff_{/N}^{\text{op}}$ to the category of mixed graded cdga over $\mathbb{k}$. 
We can derive this functor on the left, by precomposing it with cofibrant replacement functor, to obtain 
a well-defined $\infty$-functor 
\[
\DR(-/N): \dAff^{\text{op}}_{/N}\longrightarrow \medg.
\]
Recall that for a mixed graded complex $F \in \epsilon-dg^{gr}$ of $\mathbb{k}$-vector spaces we have assoicated  complexes $NC^n(F)(p) := \underset{i\geq 0}\prod F^{n-2i}(p+i) \in dg$ and a graded complex 
$NC^w(F):= \underset{p\in \mathbb Z}\oplus NC(F)(p) \in  dg^{gr}$ of negative cyclic chains. 
Recall also that for any complex $E \in dg$ of $\mathbb{k}$-vector spaces we get a simplicial set $|E| = \text{Map}_{dg}(\mathbb{k},E) =$ the Dold-Kan functor applied to the $\tau_{\leq 0}$ truncation of $E$.

\smallskip

\begin{Def} 
For $A \in \cdga^{\leq 0}$ with a map $\Spec A\longrightarrow N$ of derived stacks and two integers $p \geq 0$ and $n \in \mathbb Z$, we define
\begin{enumerate}
\item The simplicial set $\mathcal A^p_N (\Spec (A), n )$ of relative \emph{\bfseries $n$-shifted $p$-forms on $\Spec A/N$} by setting $\mathcal A^p_N (\Spec (A), n ) := \left| \wedge^p  \mathbb L_{\Spec (A)/N}[n]\right|$.

\item The simplicial set $\mathcal A^{p,cl}_N(\Spec (A),n)$ of relative  \emph{\bfseries closed $n$-shifted $p$-forms on $\Spec A/N$} by setting  $\mathcal A^{p,cl}_N(\Spec (A),n) := | NC^w(DR(\Spec (A)/N))[n - p](p)|$.
\end{enumerate}
\end{Def}

\smallskip

\noindent

Again by precomposing with cofibrant replacements, the constructions $\mathcal A^p_N ( -, n )$ and $\mathcal A^{p,cl}_N(-,n)$ give rise to $\infty$-functors \linebreak $\dAff^{op}_{/N}\longrightarrow \mathbb{S} $ from the $\infty$-category  $\dAff^{op}_{/N}
$ of affine derived schemes over $N$  to the $\infty$-category $\mathbb{S}$ of simplicial sets. We denote these $\infty$-functors by $\boldsymbol{\mathcal A^p_N ( -, n )}$ and $\boldsymbol{\mathcal A^{p,cl}_N(-,n)}$, respectively. 

\smallskip

\begin{Def}
For $\boldsymbol{\Spec} (A) \in \dAff_{/N}$, the simplicial set $\mathcal A^p_N(\boldsymbol{\Spec} (A),n)$ (respectively, $\mathcal A^{p,cl}_{N}(\boldsymbol{\Spec} (A),n)$) is called the space of $p$-forms of degree $n$ on the derived stack $\boldsymbol{\Spec} (A)$, relative to $N$ (respectively, the space of closed $p$-forms of degree $n$ on the derived stack $\boldsymbol{\Spec} (A)$, relative to $N$).
\end{Def}

\smallskip

\noindent
The two $\infty$-functors $\boldsymbol{\mathcal A^p_{N}(-,n)}$ and $\boldsymbol{\mathcal A^{p,cl}_N(-,n)}$ are derived prestacks on $\dAff_{/N}$. By \cite[Proposition 1.11]{PTVV}, we know that $\boldsymbol{\mathcal A^p_N(-,n)}$ and $\boldsymbol{\mathcal A^{p,cl}_{N}(-,n)}$ are derived stacks for the big \'{e}tale site over $N$. 

\smallskip

Now let $S$ be an l.f.p. derived Artin stack which is equipped with a locally free log strurcture. 
Let $\mathbf{dSt}_{S,\log}$ be the $\infty$ category of derived Artin stacks $M$ equipped with a locally free log structure and a logarithmic map   $M \to S$. In order to define relative shifted log $p$-forms for an object $M\longrightarrow S$ in $\mathbf{dSt}_{S,\log}$, we set  $N:=\mathfrak L^0_{S, \mathsf{lf}}$ in the above construction. In other words we define stacks $\boldsymbol{\mathcal A}^{p}_{\mathfrak L^0_{S, \mathsf{lf}}}(-, n)$, respectively $\boldsymbol{\mathcal A}^{p,cl}_{\mathfrak L^0_{S, \mathsf{lf}}}(-, n)$ of relative log $p-$forms, respectively closed log $p-$ forms,  of degree $n$ on $M \to S$ by setting $\boldsymbol{\mathcal A}^{p}_{S, \log}(-, n):= \boldsymbol{\mathcal A}^{p}_{\mathfrak L^0_{S, \mathsf{lf}}}(-, n)$, respectively $\boldsymbol{\mathcal A}^{p,cl}_{S, \log}(-, n):= \boldsymbol{\mathcal A}^{p,cl}_{\mathfrak L^0_{S, \mathsf{lf}}}(-, n)$.

\

\begin{Def}\label{slogsym}
\begin{enumerate}
\item[(1)] The \emph{\bfseries space of $S$ relative logarithmic $n$-shifted \linebreak $p$-forms on $M \in \dSt_{S, \log}$} is defined to be 
\[
\mathcal{A}_{log}^p (M/S, n):=\map_{_{\textbf{dSt}_{S, log}}}\left(M, \boldsymbol{\mathcal A}^p_{\mathfrak L^0_{S, \mathsf{lf}}}(-, n)\right).
\]
\item[(2)] The \emph{\bfseries space of closed relative logarithmic $n$-shifted $p$-forms on \linebreak $M \in \dSt_{S, \log}$} is defined to be  
\[
\mathcal{A}_{\log}^{p,cl} (M/S, n):=\map_{_{\textbf{dSt}_{S, log}}}\left(M, \boldsymbol{\mathcal A}^{p,cl}_{\mathfrak L^0_{S, \mathsf{lf}}}(-, n)\right).
\]
\item[(3)] A $2$-form $\omega\in \mathcal{A}_{\log}^2(M/S,n)$ is \emph{\bfseries non-degenerate} if the induced map in $\mathrm{D}_{qcoh}(M)$
\begin{equation}
\Theta_{\omega}: \mathbb T^{\log}_{M/S}\longrightarrow \mathbb L^{\log}_{M/S}[n]
\end{equation}  
is a quasi-isomorphism. We denote by $\mathcal{A}_{\log}^2 (M/S, n)^{nd}$ the union of all the connected components of  $\mathcal{A}_{\log}^2 (M/S, n)$ which is consist of non-degenerate relative logarithmic $2$-forms of degree $n$ on $M$. Here, $\mathbb T^{\log}_{M/S}$ denotes the dual of $\mathbb L^{\log}_{M/S}$, and we call it the relative logarithmic tangent complex. 

\item[(4)] Finally, we define the \emph{\bfseries space of $n$-shifted relative log symplectic forms} as the homotopy fibre product 
\[
Symp_{\log}(M/S, n):=
\mathcal{A}_{\log}^2 (M/S, n)^{nd}\times^{h}_{\mathcal{A}_{\log}^2 (M/S, n)} \mathcal{A}_{\log}^{2,cl} (M/S, n).
\]
\end{enumerate}
\end{Def}

\

\section{\textbf{The logarithmic Dolbeault moduli stack and its log symplectic structure}}
In this section, we will construct a logarithmic version of the relative Dolbeault moduli stack for the family of curves $\mathfrak X_{\mathfrak M}\rightarrow \mathfrak M$ defined at the end of section~\ref{Fam}. We will show that the relative logarithmic Dolbeault moduli stack has a relative $0$-shifted log-symplectic form over the spectrum of a discrete valuation ring $S$. Moreover, we will show that the relative log-symplectic form is an extension of Hitchin's symplectic form on the generic fibre of the moduli stack over the spectrum of a discrete valuation ring. This was proved for moduli schemes in \cite{D}.

\smallskip

\begin{Def}
Let $\mathcal X\longrightarrow \mathcal{S}$ be a flat representable morphism of classical Artin stacks of relative dimension $1$. We call it a \emph{\bfseries semi-stable family of curves} if every geometric fibre of $\mathcal X\longrightarrow \mathcal{S}$ is isomorphic to a semi-stable curve i.e., it satisfies the following properties
\begin{enumerate}
\item every geometric fibre is a reduced, connected, projective, nodal curve
\item every component $E$ of a geometric fiber which isomorphic to $\mathbb P^1$ must intersect 
the union of all other components at at least two smooth points.   
\end{enumerate}
\end{Def}

\smallskip

\noindent
Recall that if $C$ is a nodal curve with a set of nodes $D$ we have an explicit formula for the dualizing sheaf of $C$. Indeed, let  $q: \widetilde{C}\longrightarrow C$ denote the normalisation and let $\widetilde D$ denote the pre-image $q^{-1}(D)$. Then the \emph{\bfseries dualizing sheaf $\omega_C$} of $C$ is the kernel  
\begin{equation}
\omega_{C} = \ker \left[q_*\Omega_{\widetilde C}^{1}(\widetilde D)\longrightarrow \underset{x\in D}\bigoplus \ \mathbb k_x\right],
\end{equation}
where 
\begin{itemize}
\item $\mathbb k_x$ denotes the sky-scraper sheaf at the point $x$.
\item for the two preimages $q^{-1}(x) = \{x^{-},x^{+}\}$ of a node $x \in D$, the map \linebreak $q_*\Omega_{\widetilde C}^{1}(x^++x^-)\longrightarrow \mathbb k_x$ is given by 
\begin{equation}\label{res908}
s \mapsto \mathsf{Res}(s; x^+)+\mathsf{Res}(s;x^-), 
\end{equation}
where $\mathsf{Res}(s; x)$ denotes the residue of a form $s$ at a point $x$.
\end{itemize}
It is straightforward to check that it is a locally free sheaf of rank one.
It is well known  that this property persists in families. That is, for a family $f : \mathcal{X} \to \mathcal{S}$ of semi-stable/ pre-stable curves there is a well defined relative dualizing sheaf $\omega_{\mathcal{X}/\mathcal{S}}$ which is a line bundle on the total space of the family \cite[109.19 The relative dualizing sheaf]{34}. In the case when $\mathcal{S}$ and $\mathcal{X}$ are both smooth, the relative dualizing sheaf is naturally identified with the sheaf of relative logarithmic  forms along the fibers of $f$. Concretely, if 
$\boldsymbol{\mathfrak{d}}_{\mathcal{S}} \subset  \mathcal{S}$ denotes the discriminant divisor of the map 
$f : \mathcal{X} \to \mathcal{S}$, and $\boldsymbol{\mathfrak{d}}_{\mathcal{X}} = f^{-1}\left(\boldsymbol{\mathfrak{d}}_{\mathcal{S}}\right)$ is the normal crossings divisor in $\mathcal{X}$ comprised of the singular fibers of $f$, then we have
\[
\omega_{\mathcal{X}/\mathcal{S}} = \Omega^{1}_{\mathcal{X}}(\log \boldsymbol{\mathfrak{d}}_{\mathcal{X}})/f^{*} \Omega^{1}_{\mathcal{X}}(\log \boldsymbol{\mathfrak{d}}_{\mathcal{S}}).
\]
Alternatively, again under the assumption that   $\mathcal{X}$ and $\mathcal{S}$ are smooth,  we can identify  the relative dualizing sheaf $\omega_{\mathcal{X}/\mathcal{S}}$ with  the  relative  cotangent sheaf $\Omega^{1}_{l\mathcal{X}/l\mathcal{S}}$ of the log smooth map $l\mathcal{X} \to l\mathcal{S}$, where $l\mathcal{X}$ and $l\mathcal{S}$ are the log stacks whose underlying stacks are  $\mathcal{X}$ and $\mathcal{S}$ and whose log structures are given by the normal crossings divisors $\boldsymbol{\mathfrak{d}}_{\mathcal{X}}$ and $\boldsymbol{\mathfrak{d}}_{\mathcal{S}}$ respectively. More generally, for any semi-stable family of curves  $\mathcal{X} \to \mathcal{S}$ for which $\mathcal{S}$ is a classical Artin stack which is locally of finite type, the construction of F.~Kato \cite{17} implies that there are canonical functorial log structures $l\mathcal{X}$ and $l\mathcal{S}$ for which the map $\mathcal{X} \to \mathcal{S}$ can be refined to a log smooth morphism $l\mathcal{X} \to l\mathcal{S}$ so that still have $\omega_{\mathcal{X}/\mathcal{S}} \cong \Omega^{1}_{l\mathcal{X}/l\mathcal{S}}$.

\smallskip

\smallskip

\begin{Def} \label{def-relative}
A \emph{\bfseries relative Higgs bundle} over a family of semistable curves $\mathcal X/\mathcal S$ is a locally free $\mathcal O_{\mathcal X}-$module $\mathcal E$ with a Higgs field \linebreak $\phi: \mathcal E\longrightarrow \mathcal E\otimes \omega_{\mathcal X/\mathcal S}$.
\end{Def}

\smallskip

\begin{Rem} The vector bundle $\mathcal E$ along with the Higgs field $\phi : \mathcal{E} \to \mathcal{E}\otimes \omega_{\mathcal{X}/\mathcal{S}}$ might also be dubbed a \emph{\bfseries relative logarithmic Higgs bundle} on $\mathcal{X}/\mathcal{S}$ since, as explained above,  such a $\phi$ take values in $\Omega^{1}_{l\mathcal{X}/l\mathcal{S}}$.  In the previous definition, we chose to call such a pair $(\mathcal{E},\phi)$  simply a relative Higgs bundle since it is just a Higgs bundle on $\mathcal{X}$ with coefficients in the relative dualizing sheaf $\omega_{\mathcal{X}/\mathcal{S}}$ of the map $\mathcal{X} \to \mathcal{S}$.
\end{Rem}

\smallskip

\noindent
Next we consider $\mathsf{T}_{\mathcal X/\mathcal S}:=\underline{Spec}_{\mathcal X} Sym_{\mathcal O_{\mathcal X}} (\omega_{\mathcal X/\mathcal S})$. This is the total space of the line bundle dual to the relative dualizing sheaf of the family of curves $\mathcal X\longrightarrow \mathcal S$.

\smallskip

\begin{Def}
Let $\widehat{\mathsf{T}_{\mathcal X/\mathcal S}}$ denote the formal completion of $\mathsf{T}_{\mathcal X/\mathcal S}$ along the zero section. Notice that it is a formal group scheme over $\mathcal X$, whereas $\mathsf{T}_{\mathcal X/\mathcal S}$ is an abelian group scheme over $\mathcal X$. 
\begin{itemize} 
\item The \emph{\bfseries (logarithmic) Dolbeault stack} of $\mathcal X/\mathcal S$ is defined as the formal quotient stack
\begin{equation}
\mathcal X_{Dol/\mathcal S}:= \left[\widehat{\mathsf{T}_{\mathcal X/\mathcal S}}\times_S \mathcal X \rightrightarrows \mathcal X\right]
\end{equation}
\item 
The \emph{\bfseries nilpotent (logarithmic) Dolbeault stack} of $\mathcal X/\mathcal S$ is defined as the quotient stack 
\begin{equation}
\mathcal X^{nil}_{Dol/\mathcal S}:=[\mathsf{T}_{\mathcal X/\mathcal S}\times_{\mathcal S} \mathcal X\rightrightarrows \mathcal X]
\end{equation}
\end{itemize}
\end{Def}

\smallskip

\smallskip

\noindent
With this notation, we have the following standard lemma

\smallskip

\begin{Lem}\label{Corr1}
\begin{enumerate}
\item $\QCoh(\mathcal X_{Dol/\mathcal{S}})\cong \Mod_{\Sym_{\mathcal O_{\mathcal X}} (\omega^{\vee}_{\mathcal X/\mathcal S})} (\QCoh(\mathcal{X})) $
\item If $(\mathcal{E},\phi)$ is a quasi-coherent relative Higgs complex on $\mathcal{X}/\mathcal{S}$ and if $E$ is the corresponding quasi-coherent complex on $\mathcal{X}_{Dol/\mathcal{S}}$ then we have an isomorphism of quasi-coherent complexes on $\mathcal{S}$: 
\[
\xymatrix@R-1pc@C-1pc{
H^{\bullet}(\mathcal X_{Dol/\mathcal{S}}/\mathcal{S}, E) \ar@{=}[d]_-{\text{def}} & \cong & H^{\bullet}_{Dol}(\mathcal X/\mathcal{S}, (\mathcal E,\phi)) \ar@{=}[d]_-{\text{def}} \\
R(\mathcal X_{{Dol}/\mathcal{S}} \to \mathcal{S})_{*} E 
& & R(\mathcal{X} \to \mathcal{S})_{*} \left[  \mathcal{E} \stackrel{\phi}{\to} E\otimes \omega_{\mathcal{X}/\mathcal{S}} \right]
}
\]
\end{enumerate}
\end{Lem}
\begin{proof}
The proof follows verbatim \cite[Proposition 5.1.2]{PS} after replacing the relative cotangent complex $\mathbb L_{\mathcal X/\mathcal S}$ with the relative dualizing sheaf $\omega_{\mathcal X/\mathcal S}$.
\end{proof}

\smallskip

\begin{Rem}\label{OCOR}
    Notice that $\mathcal X_{Dol/\mathcal S}$ is the relative classifying stack $B\widehat{\mathsf{T}}_{\mathcal X/\mathcal S}$ of the formal  commutative group scheme $\widehat{\mathsf{T}}_{\mathcal X/\mathcal S}$ over $\mathcal X$. The nilpotent Dolbeault stack $\mathcal X^{nil}_{Dol/\mathcal S}$ is the classifying stack of the abelian group scheme $\mathsf{T}_{\mathcal X/\mathcal S}$ over $\mathcal X$. Therefore we have the following equivalence of categories.

    \begin{equation}\label{equi1112}
\left\{
\begin{array}{@{}ll@{}} 
\text{Quasi-coherent sheaves}\\
\text{over}~ B\widehat{\mathsf{T}}_{\mathcal X/\mathcal S}
\end{array}\right\}
\cong \left\{
\begin{array}{@{}ll@{}} 
\text{Quasi-coherent sheaves over}~\mathcal X~\text{with}\\
\text{an action of the sheaf of algebras} \\
\Sym (\omega^{\vee}_{\mathcal X/\mathcal S})
\end{array}\right\}
\end{equation}

\smallskip

Notice that $\Sym (\omega^{\vee}_{\mathcal X/\mathcal S})$ is a quadratic algebra. The quadratic dual algebra $(\Sym ~\omega^{\vee}_{\mathcal X/\mathcal S})^!$ is isomorphic to the dg-algebra $\Sym(\omega_{\mathcal X/\mathcal S}[-1])$ with zero differential. Moreover, we have the following equivalence of categories (\cite{PP}).

  \begin{equation}\label{equi1111}
\left\{
\begin{array}{@{}ll@{}} 
\text{Quasi-coherent sheaves over}~\mathcal X~\text{with}\\
\text{an action of the sheaf of algebras} \\
\Sym (\omega^{\vee}_{\mathcal X/\mathcal S})
\end{array}\right\}
\cong
\left\{
\begin{array}{@{}ll@{}} 
\text{Quasi-coherent sheaves over}\\
\Spec_{\mathcal{X}} ~ \Sym~(\omega_{\mathcal X/\mathcal S}[-1])
\end{array}\right\}
\end{equation}

From \eqref{equi1112} and \eqref{equi1111}, we have the following.
\begin{equation}\label{equi1}
B\widehat{\mathsf{T}}_{\mathcal X/\mathcal S} \cong \Spec_{\mathcal{X}}~\Sym(\omega_{\mathcal X/\mathcal S}[-1]). 
\end{equation}
Similarly, we have the following equivalence. 
\begin{equation}
\left\{
\begin{array}{@{}ll@{}} 
\text{Vector bundles}\\
\text{over}~ B\widehat{\mathsf{T}}_{\mathcal X/\mathcal S}
\end{array}\right\}
\cong \left\{
\begin{array}{@{}ll@{}} 
\text{Vector bundles over}~ \mathcal X ~~\text{with an}\\
\text{ action of the sheaf of algebras} \\
({\Sym (\omega_{\mathcal X/\mathcal S}}))^{\vee}\cong \widehat{\Sym}(\omega^{\vee}_{\mathcal X/\mathcal S})
\end{array}\right\}
\cong \left\{
\begin{array}{@{}ll@{}} 
\text{nilpotent Higgs}\\
\text{bundles over}~\mathcal X
\end{array}\right\}
\end{equation}
\end{Rem}

\smallskip

\begin{Thm}\label{main}
$\mathcal X_{Dol/\mathcal{S}}$ is $\mathcal O$-compact and $\mathcal O$-oriented over $\mathcal S$. Hence, $\map_{\mathcal S}(\mathcal X_{Dol/\mathcal{S}}, BGL_{n}(\mathcal{O}_{\mathcal{S}}))$ has a $0$-shifted relative symplectic structure over $\mathcal S$.
\end{Thm}

\smallskip

\begin{proof} \textsf{\underline{Proof of $\mathcal O$-compactness:}} Let us begin with recalling the notations for the following maps from Section \ref{ConNot}. 
\begin{equation}
\begin{tikzcd}
\mathcal X_{Dol/\mathcal{S}}\arrow["q"]{r}\arrow[bend right=20,swap, "p"]{rr}  & \mathcal X \arrow["r"]{r} & \mathcal S
\end{tikzcd}
\end{equation}
From \eqref{equi1}, we have 
\begin{equation}\label{eq0012}
 Rq_*\mathcal O_{\mathcal X_{Dol/\mathcal{S}}}\cong \mathcal O_{\mathcal X}\oplus \omega_{\mathcal X/\mathcal S}[-1], 
\end{equation}
which is clearly a perfect complex. Therefore, $\mathcal O_{\mathcal X_{Dol/\mathcal{S}}}$ is $\mathcal O$-compact relative to $\mathcal{S}$.

\smallskip

\noindent
\textsf{\underline{Proof of $\mathcal O$-orientation:}} Using \eqref{eq0012}, we have
\begin{align*}
( Rp_*\mathcal O_{\mathcal X_{Dol/\mathcal{S}}})^{\vee}[-2] & \cong ( Rr_*\mathcal O_{\mathcal X}\oplus  Rr_*\omega_{\mathcal X/\mathcal S}[-1])^{\vee}[-2] \\ & \cong ( Rr_*\mathcal O_{\mathcal X})^{\vee}[-2]\oplus ( Rr_*\omega_{\mathcal X/\mathcal S}[-1])^{\vee}[-2].
\end{align*} 
By Serre duality, we have 
\begin{align*}
 Rr_*(\omega_{\mathcal X/\mathcal S}[-1])\cong ( Rr_*((\omega_{\mathcal X/\mathcal S}[-1])^{\vee}\otimes \omega_{\mathcal X/\mathcal S}[1]))^{\vee}\cong ( Rr_*(\mathcal O_{\mathcal X}[2]))^{\vee}\cong (Rr_*\mathcal O_{\mathcal X})^{\vee}[-2]\\
\implies ( Rr_*(\omega_{\mathcal X/\mathcal S}[-1]))^{\vee}[-2]\cong  Rr_*(\mathcal O_{\mathcal X}[2])[-2]\cong  Rr_*\mathcal O_{\mathcal X} \hspace{6em}
\end{align*} 
Notice that $H^0(Rr_*\mathcal O_{\mathcal X})\cong \mathcal O_{\mathcal S}$. We choose an isomorphism once and for all and denote it by $\eta$.
This defines an element 
\begin{equation}
\eta \in ( Rp_*\mathcal O_{\mathcal X_{Dol/\mathcal{S}}})^{\vee}[-2]=R\mathcal{H}om_{\mathcal{S}}(Rp_*\mathcal O_{\mathcal X_{Dol/\mathcal{S}}}, \mathcal O_{\mathcal S}[-2]).
\end{equation}
We want to show that it is an $\mathcal O-$orientation.

Let $\mathcal E$ be a perfect complex on $\mathcal X_{Dol/\mathcal{S}}\times_{\mathcal S} \textbf{Spec}(A)$, for an $A\in \textbf{cdga}^{\leq 0}_{{\mathcal S}}$. We have to show that the following map 
\begin{equation}\label{GSD}
(-\cap \eta_A): R{p_A}_*\mathcal E\longrightarrow (R{p_A}_*(\mathcal E^{\vee}))^{\vee}[-2]=(R{p_A}_*(\mathcal E^{\vee})[2])^{\vee}
\end{equation}
is a quasi-isomorphism of $A$-dg-modules, where $\eta_A:  Rp_{A*}\mathcal O_{\mathcal X_{Dol, A}}\longrightarrow A[-2]$ is the derived pullback of $\eta:  Rp_*\mathcal O_{\mathcal X_{Dol/\mathcal{S}}}\longrightarrow \mathcal O_{\mathcal S}[-2]$ under the map $\spec (A)\longrightarrow \mathcal S$ and $(R{p_A}_*(\mathcal E^{\vee}))^{\vee}$ is the derived $A$-dual of $R{p_A}_*(\mathcal E^{\vee})$. Here, 
\[
p_A: \mathcal X_{Dol/\mathcal{S}}\times_{\mathcal S} \textbf{Spec}(A) \longrightarrow \textbf{Spec}(A)
\] is the projection. The left hand side is isomorphic to $ R{r_A}_*\mathsf{Dol}_{(E_A, \phi_A)}$, where 
\begin{align*}
\mathsf{Dol}_{(E_A, \phi_A)}:=[E_A\xrightarrow{\phi_A} E_A\otimes \omega_{\mathcal X_A, A}] \\
0 \hspace{4cm} 1
\end{align*}
corresponding to $R{q_A}_*\mathcal E_A$ (here, $E_A$ is in degree $0$). The right hand side is isomorphic to $( R{r_A}_*(\mathsf{Dol}^{\mathsf{D}}_{(E_A, \phi_A)}))^{\vee}$, where 
\begin{align}\label{0002}
    \mathsf{Dol}^{\mathbf{D}}_{(E_A, \phi_A)}:=[E^{\vee}_A\xrightarrow{-\phi^{\vee}_A} E^{\vee}_A\otimes \omega_{\mathcal X_A, A}]\\
    -2 \hspace{3.5cm} -1
\end{align}   
where $\phi^{\vee}_A$ is the image of $\phi_A$ under the natural isomorphism $\End (E_A)\cong \End (E^{\vee}_A)$. Now since the Grothendieck-Serre dual of the complex $\mathsf{Dol}_{(E_A, \phi_A)}$ is the same as $\mathsf{Dol}^{\mathbf{D}}_{(E_A, \phi_A)}$, the morphism \eqref{GSD} is quasi-isomorphism. (We have used the notation $\mathsf{Dol}^{\mathbf{D}}_{(E_A, \phi_A)}$ in \eqref{0002} to disambiguate between the ordinary dual and the Serre dual of the complex  $\mathsf{Dol}_{(E_A, \phi_A)}$.)

\smallskip

Therefore, we have shown that $\mathcal X_{Dol/\mathcal{S}}$ is $\mathcal O$-compact and $\mathcal O$-oriented over ${\mathcal S}$ and from \cite[Theorem 2.5]{PTVV}, it follows that $\map_{\mathcal S}(\mathcal X_{Dol/\mathcal{S}}, BGL_{n}(\mathcal{O}_{\mathcal{S}}))$ has a $0$-shifted relative symplectic structure over ${\mathcal S}$.
\end{proof}

\medskip

\begin{Rem} \label{Remqq}
The above theorem is more broadly applicable to any given family of reduced Gorenstein projective algebraic curves $(\mathcal{X}/S$. The main focus of the theorem is on two properties: $\mathcal{O}$-compactness and $\mathcal{O}$-orientedness of the family $\mathcal{X}_{\text{Dol}}/S$. These properties hold in this context due to the correspondences discussed in Remark \ref{OCOR}. 

Firstly, these correspondences apply to any geometric Derived stacks, as noted in \cite[Section 5.1]{PS}. In particular, they will apply to $\mathcal{X}_{\text{Dol}}/S$ for any given family of projective algebraic curves $\mathcal{X}/S$. However, the proof also depends on the fact that the dualizing sheaf is locally free and that Grothendieck-Serre duality holds. For this reason, the Gorenstein condition is essential.
\end{Rem}

\bigskip

Fix an integer $n$, which will represent rank in the moduli problem. Let $S$ be the spectrum of a discrete valuation ring on $\mathbb k$. We start with the set-up as in \S \ref{DegeCourbes}, i.e., with a semistable curve $\mathcal{X} \to S$ whose generic fiber is smooth and whose closed fiber has a single node. Let us denote the relative dualizing sheaf by $\omega_{\mathcal X/S}$. Let $\mathfrak M$ denote the stack of expanded degenerations of $\mathcal{X}/S$  bounded by the integer $n$. We denote the universal curve over $\mathfrak M$ by $\mathfrak{X}$ (see subsection \ref{Fam}). We drop the subscript $\mathfrak M$ from $\mathfrak X_{\mathfrak M}$, for simplicity. We denote the relative logarithmic Dolbeault stack of $\mathfrak{X}/\mathfrak{M}$  by $\mathfrak X_{Dol/\mathfrak{M}}$.

\

\begin{Prop}\label{qs11}
The morphism $\map_{_{\mathfrak M}}(\mathfrak{X}_{Dol/\mathfrak{M}}, BGL_n(\mathcal{O}_{\mathfrak M})) \longrightarrow \mathfrak M$ is a quasi-smooth morphism of derived Artin stacks.
\end{Prop}

\begin{proof}
Let us denote $\map_{_{\mathfrak M}}(\mathfrak{X}_{Dol/\mathfrak{M}}, BGL_n(\mathcal{O}_{\mathfrak{M}}))$ by $M$, for simplicity of notation. As before, the natural map $p : \mathfrak{X}_{Dol,\mathfrak{M}} \to \mathfrak{M}$ factors as 
\begin{equation} \label{eq:XDolM}
\begin{tikzcd}
\mathfrak{X}_{Dol/\mathfrak{M}}\arrow["q"]{r}\arrow[bend right=20,swap, "p"]{rr}  & \mathfrak{X} \arrow["r"]{r} & \mathfrak{M}
\end{tikzcd}
\end{equation}
The relative tangent complex of the morphism $M\longrightarrow \mathfrak M$ is given by 
\begin{equation}
\mathbb T_{M/\mathfrak M}\cong  Rp_*\End ~\mathcal E[1]\cong   Rr_*\mathcal C(E, \phi)[1]
\end{equation}
where $\mathcal C(E, \phi)$ denotes the following complex
\begin{equation}\label{TC113}
\left[ \End (E)\xrightarrow{[-, \phi]} \End (E)\otimes \mathsf{pr}_{\mathfrak{X}}^{*}\omega_{\mathfrak{X}/\mathfrak{M}}
\right]
\end{equation} 
with $\End (E)$ sitting in the degree $0$. Here $\mathcal E$ denotes the universal sheaf on \linebreak $\mathfrak{X}_{Dol/\mathfrak{M}}\times_{\mathfrak{M}} M$ and $(E, \phi)$ denotes the corresponding universal relative Higgs complex on $\mathfrak{X}\times_{\mathfrak{M}} M$. From \cite[Theorem 0.3]{To}, it follows that $\mathbb L_{M/\mathfrak M}$ is a perfect complex because the map $r: \mathfrak{X} \longrightarrow \mathfrak{M}$ is a proper representable local complete intersection morphism. Also, using Grothendieck-Serre duality we can see that the complex $Rr_*\mathcal C(E, \phi)[1]$ has amplitude in $[-1, 1]$. Therefore, we conclude that the morphism 
\[
\map_{_{\mathfrak M}}(\mathfrak{X}_{Dol/\mathfrak{M}}, BGL_n(\mathcal{O}_{\mathfrak M}))\longrightarrow \mathfrak M
\] is a quasi-smooth morphism of derived Artin stacks.
\end{proof}
\smallskip

\begin{Rem}
    It is important to note that for the proof of the proposition to be valid, we required two key criteria regarding the family of curves $\mathfrak{X} \to \mathfrak{M}$: properness and the local complete intersection property (e.g. nodal or cuspidal singularities). These criteria ensure that the derived pushforward of any perfect complex remains perfect and that the relative version of Grothendieck-Serre duality holds for such a family of curves. 
\end{Rem}

\smallskip

\begin{Thm}\label{main112}
The derived Artin stack $\map_{_{\mathfrak M}}(\mathfrak{X}_{Dol/\mathfrak{M}}, BGL_n(\mathcal{O}_{\mathfrak M}))$ has a natural relative $0$-shifted log-symplectic structure over $S$. Moreover, this form coincides with the relative log-symplectic form described in \cite{D}.
\end{Thm}

\smallskip

\begin{proof}
Recall that $\mathfrak{L}^{0}_{S,\mathsf{lf}}$ denotes the classifying stack of log morphisms to $S$ from schemes equipped with locally free log structures. The log structure on $\mathfrak M$, which we discussed in Proposition \ref{rellog101}, then gives a map $\mathfrak M\to \mathfrak{L}^{0}_{S,\mathsf{lf}}$. 

Consider the composite morphism. 
\[
M:=\map_{_{\mathfrak M}}(\mathfrak{X}_{Dol/\mathfrak{M}}, BGL_n(\mathcal{O}_{\mathfrak M})) \longrightarrow \mathfrak M\longrightarrow \mathfrak{L}^{0}_{S,\mathsf{lf}}
\] 
The composite morphism induces a log structure on $M$, which is the same as the pull-back of the log structure from $\mathfrak M$. Notice that the map $M \to \mathfrak{M}$  is quasi-smooth (Proposition~\ref{qs11}) and the map 
$\mathfrak{M} \to \mathfrak{L}^{0}_{S,\mathsf{lf}}$ is smooth with $\mathbb L_{\mathfrak M/\mathfrak{L}^{0}_{S,\mathsf{lf}}}\cong 0$ (Proposition \ref{rellog101}). Therefore, it follows that 

\begin{equation}\label{TC112}
\mathbb L_{M/\mathfrak{L}^{0}_{S,\mathsf{lf}}}\cong \mathbb L_{M/\mathfrak M}\cong  Rr_*\mathcal C(E, \phi)[1].
\end{equation}

So, the relative log-cotangent complex of the map $M\longrightarrow S$ is isomorphic to the relative cotangent complex of the map $M\longrightarrow \mathfrak M$. From Theorem \ref{main}, it follows that the stack $M$ has a $0$-shifted symplectic structure relative to the stack $\mathfrak M$, which is, by definition \ref{slogsym}, a $0$-shifted log-symplectic structure on $M$ relative to $S$. 

The log-symplectic pairing can be described as follows. The stack $BGL_n$ has a $2$-shifted symplectic form given by the trace pairing:
\begin{equation}\label{BG}
\mathbb{T}_{BGL_n}\wedge \mathbb{T}_{BGL_n}=\mathfrak gl_n[1]\wedge \mathfrak gl_n[1]\xrightarrow{(\mathsf x,\mathsf y)\mapsto Trace(\mathsf x\cdot \mathsf y)} \mathbb k[2].
\end{equation}
Now let $A$ be an object of $\cdga^{\leq 0}$ and let $\Spec(A) \to M$ be an $A$-valued point of $M$. Write $\mathfrak{X}_{A}$ for the pullback $\mathfrak{X}\times_{\mathfrak{M}} \Spec(A)$ of the universal curve to 
$\Spec(A) \to M \to \mathfrak{M}$. We will also write 
$\mathfrak{X}_{Dol, A}$ for the pullback $\mathfrak{X}_{Dol/\mathfrak{M}}\times_{\mathfrak{M}} \Spec(A)$ for the pullback of the relative Dolbeault stack. Now, by the definition of $M$, the map $\Spec(A) \to M$ gives us a map \[f_A: \mathfrak{X}_{Dol,A}\longrightarrow BGL_n\times \Spec (A)\] which corresponds to a vector bundle $\mathcal{E}_A$ on $\mathfrak X_{Dol, A}$. The vector bundle $\mathcal E_A$ on $\mathfrak X_{Dol, A}$ corresponds to a Higgs field \[
(E_{A} \xrightarrow{\phi_A} E_{A}\otimes \omega_{\mathfrak{X}_A , A}).
\]
on $\mathfrak X_A$.

We have the following induced pairing by pulling back \eqref{BG}.
\begin{equation}
    f^*_A \mathbb{T}_{BGL_n}\otimes f^*_A \mathbb{T}_{BGL_n}= \End ~\mathcal{E}_A[1]\otimes \End ~\mathcal{E}_A [1]\xrightarrow{Tr:=Trace} \mathcal O_{\mathfrak{X}_{Dol, A}}[2]
\end{equation}
Now by pushing it forward by the map $q_A$, we get 
\begin{equation}\label{pair556}
\begin{tikzcd}[scale=0.1]
   {q_A}_*(\End ~E_A[1] \otimes  \End ~E_A[1])\arrow{rr} &&  (\mathcal O_{\mathcal X_A}\oplus \omega_{\mathfrak{X}_A, A}[-1])[2]
\end{tikzcd}
\end{equation}
The complex  ${q_A}_*(\End ~\mathcal{E}_A [1])=[\End ~E_A \xrightarrow{[-,\phi_A]} \End ~E_A \otimes \omega_{\mathfrak{X}_A, A}]$ can be used to describe the pullback of the pairing in \eqref{pair556} by the natural map 
\[
{q_A}_*(\End ~\mathcal{E}_A [1]) \otimes {q_A}_*(\End ~\mathcal{E}_A [1]) \to {q_A}_*(\End ~\mathcal{E}_A[1] \otimes \End ~\mathcal{E}_A [1])
\]
as the following map of complexes.
\begin{small}
\begin{equation}
    \begin{tikzcd}[scale cd=.85]
    -2 & -1 & 0\\
        (\End ~E_A)^{\otimes 2} \arrow["d_1"]{r} \arrow["Tr"]{d} & \End ~E_A \otimes (\End ~E_A \otimes \omega_{\mathfrak{X}_A, A})\oplus  (\End ~E_A \otimes \omega_{\mathfrak{X}_A, A})\otimes \End ~E_A \arrow["d_2"]{r} \arrow["Tr"]{d} & (\End ~E_A \otimes \omega_{\mathfrak{X}_{A}, A})^{\otimes 2} \arrow{d}\\
        \mathcal O_{\mathfrak{X}_A} \arrow["0"]{r} & \omega_{\mathfrak{X}_{A}, A} \rar & 0
    \end{tikzcd}
\end{equation}
\end{small}
where the map $d_1:=([-, \phi_A], [-, \phi_A])$ and $d_2:=[-, \phi_A]\otimes \mathbb 1+ \mathbb 1\otimes [-, \phi_A]$.
Finally, by pushing forward \eqref{pair556} by the map $r_A$ and then composing with the orientation $\eta_A$ (see \eqref{GSD}), we get our log-symplectic pairing. From the above description it is evident that this log-symplectic form coincides with the log-symplectic form discussed in \cite{D}.
\end{proof}

\begin{Rem}
    Note that the proof of the first statement of the above Theorem depends on the following crietria: (i) there is a log-structures on the space $\mathfrak M$, (ii) the map $M\to \mathfrak M$ is quasismooth, (iii) there is a relative 0-shifted symplectic form on $M\to \mathfrak M$ and finally the fact that (iv) $\mathbb L_{\mathfrak M/\mathfrak{L}^{0}_{S,\mathsf{lf}}}\cong 0$. With these four hypothesis, we have shown that there is a relative 0-shifted shifted symplectic form on $M\to S$. The second statement of the theorem is about the specific situtation appearing from a one parameter degeneration of a family of stable curves.  
\end{Rem}

\section{\textbf{Completeness of the Hitchin map}}
In this section, we define the Hitchin map on the ordinary Artin stack of Gieseker-Higgs bundles $M^{cl}_{Gie}$ (see \ref{Gies21}). We prove that the Hitchin map is complete. This result was proved in \cite{2a} for the Hitchin map as defined on the moduli scheme of stable Higgs bundles in the case where rank and degree are co-prime. We prove it here for the moduli stack, and the argument does not require us to assume that the rank and degree are co-prime.

\

\begin{Def} \cite[p7]{33}
Working over $\Spec \mathbb{k}$, a morphism of Artin stacks $f: \mathcal Y_1\longrightarrow \mathcal Y_2$ is called \textbf{complete} if given any discrete valuation ring $A$ over $\mathbb{k}$ with fraction field $K$ and any commutative diagram 
\begin{equation}
\begin{tikzcd}\label{eq0234}
\Spec~K\arrow{r} \arrow{d} & \mathcal Y_1\arrow["f"]{d}\\
\Spec~A\arrow{r}& \mathcal Y_2
\end{tikzcd}
\end{equation}
there exists finite field extension $K\longrightarrow K'$ with $A'$ the integral closure of $A$ in $K'$ and a lift over (the dotted arrow) $\Spec~A'\longrightarrow \mathcal X$ making all the triangles in the diagram commute.
\begin{equation}
\begin{tikzcd}
\Spec~K'\arrow{r} \arrow{d} & \Spec~K\arrow{r} \arrow{d} & \mathcal Y_1\arrow["f"]{d}\\
\Spec~A'\arrow{r} \arrow[dotted]{urr} & \Spec~A\arrow{r} & \mathcal Y_2
\end{tikzcd}
\end{equation}
A morphism of derived Artin stacks is called \textbf{complete} if the underlying morphism of ordinary Artin stacks is complete.
\end{Def}

\smallskip

\begin{Rem}
    Notice that from the definition it is obvious, in particular, that the fibers of a complete morphism of Artin stacks are complete because the morphism in the diagram \ref{eq0234} is allowed to factor through the map $\Spec~A\to \Spec k\to \mathcal Y_2$. The notion of completeness is the same as the notion of uiversal closedness of the morphism of stacks (\cite[Proposition 26.20.6 (Valuative criterion of universal closedness).]{34}). 
\end{Rem}

We recall the following setup from before. Let $S$ be a spectrum of a discrete valuation ring over $\mathbb k$. We start with the set-up as in \S \ref{DegeCourbes}, i.e., with a semistable curve $\mathcal{X} \to S$ whose generic fiber is smooth and whose closed fiber has a single node. Let us denote the relative dualising sheaf by $\omega_{\mathcal X/S}$. Let $\mathfrak M$ denote the stack of expanded degenerations of $\mathcal{X}/S$  bounded by the integer $n$. We denote the universal curve over $\mathfrak M$ by $\mathfrak{X}$ (see subsection \ref{Fam}). We drop the subscript $\mathfrak M$ from $\mathfrak X_{\mathfrak M}$, for simplicity. We denote the relative logarithmic Dolbeault stack of $\mathfrak{X}/\mathfrak{M}$  by $\mathfrak X_{Dol/\mathfrak{M}}$.

\

\begin{Def}\textbf{(Hitchin map)}\label{h11}
We recall the following two well-known definitions for the family of curves $\pi: \mathcal X \to S.$
\begin{enumerate}
\item $B:=\mathrm{Tot}(\oplus^n_{i=1} R^0 \pi_* \omega_{\mathcal X/S})$. It is a vector bundle over the spectrum $S$ of a discrete valuation ring over $\mathbb{k}$. As is standard, we call $B$ the Hitchin base. 
\item There is a natural map \[h: M:=\map_{_{\mathfrak M}}(\mathfrak{X}_{Dol/\mathfrak{M}}, BGL_n(\mathcal{O}_{\mathfrak M}))\longrightarrow B,\] which sends a family of Higgs bundles $(\mathfrak X_T, \mathcal E_T, \phi_T)$ parametrised by an affine scheme $T$ over $S$ to \[(- Trace(\phi_T), Trace(\wedge^2 \phi_T),\dots , (-1)^i Trace(\wedge^i\phi_T),\dots,(-1)^n Trace(\wedge^n \phi_T)). \]
inside $T\times_{S} B$
\end{enumerate}
\end{Def}

\

Now consider the substack $M^{cl}_{Gie} \subseteq \tau_0(M) \subseteq M$ in the classical truncation $\tau_0(M)$ of $M$ parameterizing Higgs bundles whose underlying vector bundle satisfies the Gieseker conditions. That is, it is generated globally on the fibers of the map $X_k\to X_0$ and its push-forward to $X_0$ is torsion-free. See \ref{Gies21} for a detailed description. 

\

The Hitchin map $h$ (see definition \ref{h11}) defines a map $M^{cl}_{Gie}\to B$ by restriction. 

\begin{Thm}\label{complete11}
The morphism $h|_{M^{cl}_{Gie}} \colon M^{cl}_{Gie} \longrightarrow B$ is complete.
\end{Thm}

\begin{proof}
    Let $  T = \Spec A  $ be the spectrum of a discrete valuation ring over $  \mathbb{k}  $, with fraction field $  K  $. Denote the generic point by $  T^o = \Spec K  $. Suppose that we are given a commutative diagram
\begin{equation}\label{ref55667}
\begin{tikzcd}
T^o\arrow{r} \arrow{d} &M^{cl}_{Gie}\arrow["h"]{d}\\
T\arrow{r}& B
\end{tikzcd}
\end{equation}
We must show that this diagram extends to a morphism $  T \to M^{cl}_{Gie}  $.
Let $  \TFH(\mathcal{X}/S)  $ be the ordinary Artin stack of families of torsion-free $  \omega_{\mathcal{X}/S}  $-valued Higgs pairs on $  \mathcal{X}/S  $ (see Definition \ref{at3}). There is a natural morphism
$$\theta \colon M^{cl}_{Gie} \longrightarrow \TFH(\mathcal{X}/S),
\qquad (\pi \colon \mathfrak{X} \to \mathcal{X},\ \mathcal{E},\ \phi) \mapsto ((\pi_* \mathcal{E}),\ (\pi_* \phi)).$$
This is induced by pushforward along the modification $  \mathfrak{X} \to \mathcal{X}  $. Applying $  \theta  $ to the original diagram gives a commutative triangle
\begin{equation}
\begin{tikzcd}
T^o\arrow{r} \arrow{d} & \TFH(\mathcal X/S)\\
T
\end{tikzcd}
\end{equation}
By \cite[Proposition 6]{La} and \cite[Lemma 6.5]{17a}, there exists a morphism $  T \to \TFH(\mathcal{X}/S)  $ making the triangle commute.
Now consider the surface $  \mathcal{X}_T := \mathcal{X} \times_S T  $. There are two cases:

\medskip

\underline{Case 1.} The map $  T \to S  $ is not faithfully flat.
Then $  \mathcal{X}_T \cong X_0 \times T  $. The normalisation of $  \mathcal{X}_T  $ is $  \widetilde{X}_0 \times T  $. Let $  f \colon \widetilde{X}_0 \to X_0  $ be the normalisation map. By \cite[Proposition 3.2]{DI}, the pullback of the generic Higgs bundle extends to a good generalized parabolic Hitchin pair on $  \widetilde{X}_0 \times T  $, which in turn gives a torsion-free Hitchin pair on $  \mathcal{X}_T  $.

\medskip

\underline{Case 2.} The map $  T \to S  $ is faithfully flat.
Then $  \mathcal{X}_T  $ is a normal surface with an isolated singularity of type $  \mathbb{k}[[x,y,t]]/(xy - t^d)  $. Let $  i \colon \mathcal{X}_{T^o} \hookrightarrow \mathcal{X}_T  $ be the inclusion of the generic fibre. Suppose we have a Higgs bundle $  (\mathcal{E}^o, \phi^o)  $ on $  \mathcal{X}_{T^o}  $. Any vector bundle on the generic fibre extends to a torsion-free sheaf $  \mathcal{F}  $ on $  \mathcal{X}_T  $. We wish to extend the Higgs field to $  \mathcal{F} \to \mathcal{F} \otimes \omega_{\mathcal{X}_T/T}  $.
Consider the diagram
\begin{equation}
\begin{tikzcd}
0\arrow{r} & \mathcal F \arrow{r} & i_*\mathcal E^o\arrow["i_*\phi^o"]{d}\arrow{r} & i_*\mathcal E^{o}/{\mathcal F} \arrow{r} & 0\\
0\arrow{r} & \mathcal F\otimes \omega_{\mathcal X_T/T} \arrow{r} & i_*(\mathcal E^{o}\otimes i^{*}\omega_{\mathcal X_T/T})\arrow{r} & (i_*\mathcal E^{o}/{\mathcal F})\otimes \omega_{\mathcal X_T/T}\arrow{r} & 0
\end{tikzcd}
\end{equation}
By \cite[Lemma 6.5]{17a}, the Higgs field extends over $  \mathcal{X}_T  $ except possibly at the node. The composite map $  \mathcal{F} \to (i_*\mathcal{E}^o / \mathcal{F}) \otimes \omega_{\mathcal{X}_T/T}  $ vanishes generically on the closed fibre. Thus the obstruction lies in the torsion part of $  i_*\mathcal{E}^o / \mathcal{F}  $ (viewed as a sheaf on its support $  X_0  $). Restricting the top sequence to the closed fibre gives
\

\begin{equation}
0\longrightarrow \Tor^1_{_{\mathcal X_T}}(i_*\mathcal E^{o}/{\mathcal F}, \mathcal O_{X_0})\longrightarrow \mathcal F|_{X_0}\longrightarrow (i_*\mathcal E^{o})|_{X_0}\longrightarrow i_*\mathcal E^{o}/{\mathcal F}\longrightarrow 0.
\end{equation}

\

\noindent
Since $  \Tor^1_{_{\mathcal X_T}}(i_*\mathcal E^{o}/{\mathcal F}, \mathcal O_{X_0})= (i_*\mathcal{E}^o / \mathcal{F}) \otimes \mathcal{O}_{\mathcal{X}_T}(-X_0)  $ and $  \mathcal{O}_{\mathcal{X}_T}(-X_0)  $ is locally free, the sheaf $  i_*\mathcal{E}^o / \mathcal{F}  $ is pure (torsion-free on its support) if and only if the Tor term is pure. But the Tor term is a subsheaf of $  \mathcal{F}|_{X_0}  $, hence torsion-free on $  X_0  $. This implies the obstruction vanishes, and the Higgs field extends everywhere.

\medskip

In both cases we obtain a commutative diagram
 \begin{equation}\label{prop112}
\begin{tikzcd}
T^o\arrow{r} \arrow{d} &M^{cl}_{Gie}\arrow["\theta"]{d}\\
T\arrow{r}& \TFH(\mathcal X/S).
\end{tikzcd}
\end{equation}
Choose a connected component $  R^{\Lambda, m_T}_S  $ of an atlas for $  \TFH(\mathcal{X}/S)  $ such that the image of $  T \to \TFH(\mathcal{X}/S)  $ lies in the image of $  R^{\Lambda, m_T}_S \to \TFH(\mathcal{X}/S)  $. Let $  \mathcal{Y}^H_S  $ be the corresponding Quot scheme of Gieseker-Hitchin pairs over $  R^{\Lambda, m_T}_S  $ (see \ref{at2}), with projection $  \tilde{\theta} \colon \mathcal{Y}^H_S \to R^{\Lambda, m_T}_S  $.
Both maps $  R^{\Lambda, m_T}_S \to \TFH(\mathcal{X}/S)  $ and $  \mathcal{Y}^H_S \to M^{cl}_{Gie}  $ are principal $  GL_N  $-bundles for sufficiently large $  N  $, and $  \tilde{\theta}  $ is $  GL_N  $-equivariant (see \ref{equivar}). 

 \begin{equation}
\begin{tikzcd}
\mathcal Y^H_S\arrow["\tilde{\theta}"]{d}[swap]{GL_N-equiv.}\arrow[]{r} \arrow{d} & M^{cl}_{Gie}\arrow["\theta"]{d}\\
 R^{\Lambda, m_T}_S \arrow[]{r} & \TFH(\mathcal X/S)
\end{tikzcd}
\end{equation}

Since $T$ is a spectrum of a discrete valuation ring, any $GL_N$ bundle on $T$ is trivial. Therefore, by choosing a trivialisation of the $GL_N$-bundle $R^{\Lambda, m_T}_S\times_{\TFH(\mathcal X/S)} T\to T$ we get the following lift of the diagram \ref{prop112}.
\begin{equation}
\begin{tikzcd}
T^o\arrow{r} \arrow{d} & \mathcal Y^H_S\arrow["\tilde{\theta}"]{d}\arrow{r} \arrow{d} & M^{cl}_{Gie}\arrow["\theta"]{d}\\
T\arrow{r}& R^{\Lambda, m_T}_S \arrow{r} & \TFH(\mathcal X/S)
\end{tikzcd}
\end{equation}
From \cite[Proposition 5.11]{2a}, it follows that the map $\tilde{\theta}: \mathcal Y^H_S\longrightarrow R^{\Lambda, m_T}_S$ is proper (we discuss the proof below in Lemma \ref{rem:complete1} and Corollary \ref{properness_short} for the reader's convenience). So, there exists a lift $T\longrightarrow \mathcal Y^H_S$. The composite map $T\longrightarrow \mathcal Y^H_S\longrightarrow M^{cl}_{Gie}$ is our desired extension. 
This completes the proof.
\end{proof}

\begin{Lem} \label{rem:complete1}
    Let $Y\longrightarrow R$ be a quasi-projective morphism of separated schemes over a spectrum of a discrete valuation ring $S$ which is an isomorphism over the generic point of $S$. Assume that both $Y$ and $R$ are flat over $S$. Suppose that any map $T\longrightarrow R$ with $T$ flat over $S$ can be lifted to a map $T\longrightarrow Y$. Then the map $Y\longrightarrow R$ is proper.
\end{Lem}

\begin{proof}
 Take the closure of $Y$ inside some projective bundle (over $R$) to get a projective morphism $\overline{Y}\longrightarrow R$. Take any element $y\in \overline{Y}\setminus Y$. Then there exists a map $T\longrightarrow \overline{Y}$ (here, $T$ is a spectrum of a d.v.r) passing through $y$ and flat over $S$. Therefore, we get a map $T\longrightarrow R$. But then, by our lifting assumption,  this map can be lifted to a map $T\longrightarrow Y$. Let $y'\in Y$ be the image of the closed point of $T$ under this map. But, by the separateness of $\overline{Y}$, we must have $y=y'$ and therefore $\overline{Y}=Y$, and the map $Y\longrightarrow R$ is proper.  
\end{proof}

\begin{Cor}\cite[Proposition 5.11]{2a}\label{properness_short}
 The map $\tilde{\theta}: \mathcal Y^H_S\longrightarrow R^{\Lambda, m_T}_S$ is proper.
\end{Cor}

\begin{proof}
    Since the map $\tilde{\theta}$ is quasi-porjective, we choose a closure $\overline{\mathcal Y^H_S}\longrightarrow R^{\Lambda, m_T}_S$ inside some projective bundle over $R^{\Lambda, m_T}_S$. Let $y\in \overline{\mathcal Y^H_S}\setminus {\mathcal Y^H_S}$. Then there exists a map $T\longrightarrow \overline{\mathcal Y^H_S}$ (here, $T$ is a spectrum of a d.v.r) passing through $y$ and flat over $S$. Consider the composite map $T\to \overline{\mathcal Y^H_S}\to R^{\Lambda, m_T}_S$. Let $(\mathcal F_T, \phi_T)$ be the flat family of torsion-free Higgs pair on $\mathcal X_T:=\mathcal X\times_S T$ corresponding to the map $T\to R^{\Lambda, m_T}_S$. 

    \
    
    \underline{Claim:} The map $T\to R^{\Lambda, m_T}_S$ can be lifted to a map $T\to \mathcal Y^H_S$. 

    \
    
    \underline{Proof of the claim:} Notice that the surface $\mathcal X_T$ may have a singularity of type $\frac{\mathbb k[|x,y,t|]}{(xy-t^n)}$ because the map $T\to S$ is a faithfully-flat map of spectrum  of discrete valuation rings. Let $r_{_T}: \mathcal X^{{res}}_T\longrightarrow \mathcal X_T$ be the minimal resolution of singularities. Then from \cite[Proposition 6.5]{Lip} and \cite[\S 4]{2a} it follows that the vector bundle $\mathcal E_T:=\frac{r_{_T}^*\mathcal F_T}{Torsion}$ has the property that $(r_{_T})_*\mathcal E_T\cong\mathcal F_T$. Notice also that by construction and by the property that $(r_{_T})_*\mathcal E_T\cong \mathcal F_T$, it follows that the natural map $r_{_T}^*(r_{_T})_*\mathcal E_T\longrightarrow \mathcal E_T$ is a surjective map of sheaves, which means $(\mathcal E_T)|_R$ is globally generated. Here $R$ denotes the chain of rational curves in $\mathcal X^{{res}}_T$. Hence, by definition (Definition \ref{Gie1101} and Remark \ref{AltGie}), the vector bundle $\mathcal E_T$ is a Gieseker vector bundle. The pullback of the Higgs field $\phi_{_T}$ defines a Higgs field on $\mathcal E_T$. The pair $(\mathcal E_T, \phi_T)$ is the desired lift. 

    Therefore from Lemma \ref{rem:complete1} it follows that the map $\tilde{\theta}: \mathcal Y^H_S\longrightarrow R^{\Lambda, m_T}_S$ is proper.

\end{proof}

\noindent

\

\section{\textbf{Flatness of the Hitchin map}}
In this section, we study the reduced global nilpotent cone of $M^{cl}_{Gie}$ (see \ref{Gies21}), which is the reduced  fibre of $h$ over the point $0\in B$. We prove that every irreducible component of of the reduced nilpotent cone has an open subset which is an isotropic substack of $M$ (the derived stack of Higgs bundles) with respect to its log-symplectic form. We use this to compute the dimension of the reduced nilpotent cone and to show that the Hitchin map is flat. 

\

\begin{Def}
The nilpotent cone is the Hitchin fibre over the zero section $0_S$ of $B\longrightarrow S$, i.e., the following fibre product (fibre product of classical Artin stacks)
\begin{equation}
\begin{tikzcd}
\mathcal Nilp:=0_S\times_{B} M^{cl}_{Gie} \arrow{r}\arrow{d} & M^{cl}_{Gie}\arrow{d}\\
0_S\arrow{r} & B 
\end{tikzcd}
\end{equation}
The fibre product is usually non-reduced. Let $\mathcal Nilp^{^{red}}$ be the reduction of $\mathcal Nilp$. It is a stack over the spectrum of a discrete valuation ring $S$. We denote its closed fibre by $\mathcal Nilp^{^{red}}_{_0}$.
\end{Def}

\

\begin{Rem}\label{Nilp12}
A point in $\mathcal Nilp^{^{red}}_{_0}$ is a tuple $(X_r, \mathcal E, \phi)$, where $X_r$ is a Gieseker curve with $r$ number of $\mathbb P^1$ bubbles $(0\leq r\leq n)$, $\mathcal E$ is a Gieseker vector bundle on $X_{r}$  and $\phi$ is a nilpotent Higgs field, that is, $\phi^n=0$. Given a nilpotent Higgs field, we get a canonical filtration by saturated torsion-free sub-sheaves 

\begin{equation}\label{fflag}
\mathcal E^0:=0\subsetneq \mathcal E^1:=\ker \phi\subsetneq \mathcal E^{2}:=\ker \phi^{2} \subsetneq \dots\subsetneq \mathcal E^k:=\ker \phi^k\subsetneq \mathcal E^{k+1}:=\mathcal E
\end{equation}
for some integer $k$ such that $\phi^{k+1}=0$. 

\end{Rem}

\

\begin{Def}\textbf{(Type of a nilpotent torsion-free Higgs pair on the nodal curve $X_0$)}
Let $(\mathcal F, \psi)$ be a torsion-free Higgs pair on $X_0$, where $\mathcal F$ is a torsion-free coherent sheaf on $X_0$ and $\psi: \mathcal F\longrightarrow \mathcal F\otimes \omega_{X_0}$ is a map of coherent sheaves. Suppose that the Higgs field is nilpotent, i.e. $\psi^n=0$, where $n$ denotes the rank of the torsion-free sheaf $\mathcal F$. Then, as mentioned above, we get a natural flag of saturated sub sheaves of $\mathcal F$.
\begin{equation}\label{fflag14}
    \mathcal F^0:=0\subsetneq \mathcal F^1:=\ker \psi\subsetneq \mathcal F^{2}:=\ker \psi^{2} \subsetneq \dots\subsetneq \mathcal F^k:=\ker \psi^k\subsetneq \mathcal F^{k+1}:=\mathcal F
\end{equation}
For every $1\leq i\leq k+1$, we define $n_i:=\rank ~\left(\frac{\mathcal F^{i}}{\mathcal F^{i-1}}\right)$ and $d_i:=\deg~ \left(\frac{\mathcal F^{i}}{\mathcal F^{i-1}}\right)$, where the degree of a torsion-free sheaf on an irreducible projective curve is defined to be its first Chern number. We say that the nilpotent Higgs bundle $(\mathcal F, \psi)$ is of type $\{(n_i, d_i)\}^{i=k+1}_{i=1}$. We denote it by $\Type~(\mathcal F, \psi)$. 
\end{Def}

\

\begin{Def}\textbf{(Type of a nilpotent Higgs bundle on a Gieseker curve $X_r$ for some $r\in [0,n]$)}\label{Type}
Let $r\in [0,n]$ be an integer. Let $\pi_r:X_r\longrightarrow X_0$ be the Gieseker curve with exactly $r$ many $\mathbb P^1$'s. Let $(\mathcal E, \phi)$ be a Gieseker-Higgs bundle on $X_r$. Then, by definition, $((\pi_r)_*\mathcal E, (\pi_r)_*\phi)$ is a nilpotent torsion-free Higgs pair on $X_0$. We define $\Type~(X_r, \mathcal E, \phi):=\Type~((\pi_r)_*\mathcal E, (\pi_r)_*\phi)$.
\end{Def}

\

\begin{Lem}\label{pushforward1122}
    Let $(\pi_r: X_r\rightarrow X_0, \mathcal E, \phi)$ be a nilpotent Gieseker-Higgs bundle. Then we have the following induced filtration as in \eqref{fflag}.
    \begin{equation}
        \mathcal E^0:=0\subsetneq \mathcal E^1:=\ker \phi\subsetneq \mathcal E^{2}:=\ker \phi^{2} \subsetneq \dots\subsetneq \mathcal E^k:=\ker \phi^k\subsetneq \mathcal E^{k+1}:=\mathcal E
    \end{equation}
    The induced nilpotent torsion-free Higgs pairs $(\mathcal F:=(\pi_r)_*\mathcal E, \psi:=(\pi_r)_*\phi)$ also has a natural filtration as in \eqref{fflag14}. 
    \begin{equation}
    \mathcal F^0:=0\subsetneq \mathcal F^1:=\ker  \psi\subsetneq \mathcal F^{2}:=\ker \psi^{2} \subsetneq \dots\subsetneq \mathcal F^k:=\ker  \psi^k\subsetneq \mathcal F^{k+1}:=\mathcal F
\end{equation}
Then $(\pi_r)_*\mathcal E^i\cong \mathcal F^i$ for every $1\leq i\leq k+1$.
\end{Lem}

\begin{proof}
    For every $i\in [1,k+1]$, we have morphisms $\phi^i: \mathcal E\longrightarrow \mathcal E\otimes \omega^{\otimes i}_{X_r}$ and $\psi^i: \mathcal F\longrightarrow \mathcal F\otimes \omega^{\otimes i}_{X_0}$. Moreover, $\ker(\phi^i)=\mathcal E^i$ and $\ker(\psi^i)=\mathcal F^i$. Let $\sigma$ be a local section of $(\pi_r)_*\mathcal E^i$ on a neighbourhood $U$ of the node of $X_0$. It is an element of $\mathcal E^i((\pi_r)^{-1}(U))$. Therefore the section $\phi^i(\sigma)=0$ on the open set $(\pi_r)^{-1}(U)\cap \tilde{X_0}=U^o$, where $U^o$ denotes the complement of the node of $X_0$. Therefore $\psi^i(\sigma)=0$ on $U^o$, because $\psi=\phi$ on $U^o$. Since $U^o$ is dense in $U$, therefore $\psi^i(\sigma)=0$ on $U$ and $\sigma\in \mathcal F^i$. Therefore, $(\pi_r)_*\mathcal E^i\subseteq \mathcal F^i$ for all $i\in [1,k+1]$.

    For the converse, let $\sigma\in \mathcal F^i(U)$. Since $\mathcal F^i\subseteq \mathcal F=(\pi_r)_*\mathcal E$, we have $\sigma \in \mathcal E((\pi_r)^{-1}(U))$. Since $\psi^i(\sigma)=0$, therefore $\phi^i(\sigma)=0$ on  $(\pi_r)^{-1}(U^o)$ and hence, by continuity, on $(\pi_r)^{-1}(U)\cap \widetilde{X_0}$. Now notice that $\phi^i(\sigma)$ is a section of $\mathcal E\otimes \omega_{X_r}$. The bundle $\mathcal E\otimes \omega_{X_r}$ is a Gieseker-vector bundle because $\omega_{X_r}|_R\cong \mathcal O_R$, where $R$ the chain of $\mathbb P^1$'s in $X_r$. Therefore, it must vanish everywhere since the section vanishes at the two points $\widetilde{X_0}\cap R$. This implies that $\sigma\in \mathcal E^i((\pi_r)^{-1}(U))=((\pi_r)_*\mathcal E^i)(U)$. Therefore, $\pi_{r*}(\mathcal E^i)=\mathcal F^i$ for all $i\in [1,k+1]$.   
\end{proof}

\

\begin{Def}
We define $\mathcal Nilp^{sm, gen}_0$ to be the open substack of $\mathcal Nilp^{red}_0$ consisting of nilpotent Higgs bundles $(X_r, \mathcal E, \phi)$ which satisfy the following two conditions.
\begin{enumerate}
    \item $(X_r, \mathcal E, \phi)$ is a smooth point of $\mathcal Nilp^{red}_0$. 
    \item the $\Type(X_r, \mathcal E, \phi)$ is the same as the $\Type$ of the generic point of the component of the reduced nilpotent cone that $(X_r, \mathcal E, \phi)$ belongs to. 
\end{enumerate}
\end{Def}

\

Let $(\mathfrak X_{_{uni}}, \mathcal E_{_{uni}}, \phi_{_{uni}})$ denote the restriction of the universal modification, universal vector bundle and the universal Higgs field to $M^{cl}_{Gie}$. We restrict it to $\mathcal Nilp^{sm, gen}_0$. For simplicity of the notation, let us drop the subscript "univ" from the notation and simply denote it by $(\mathfrak X, \mathcal E, \phi)$.

\begin{Lem}\label{flat23}
 Consider the filtration \ref{fflag} of $\mathcal E$ induced by the Higgs field $\phi$. Then the sheaves $\mathcal E^i$ in the filtration are all flat over $\mathcal Nilp^{sm, gen}_0$. 
\end{Lem}

\begin{proof}
Consider any connected component $C$ of $\mathcal Nilp^{sm, gen}_0$ to see this. Then we have the following exact sequence for any element $c:\spec \mathbb k\longrightarrow C$.
$$
0\longrightarrow \textsf K_c^i\longrightarrow c^*\mathcal E\xrightarrow{(c^*\phi)^i} c^*\mathcal E\otimes \omega^{\otimes i}_{c^*\mathfrak X}\longrightarrow \textsf{CK}^i_c\longrightarrow 0
$$
Here $\textsf K_c^i$ and $\textsf{CK}^i_c$ denote the kernel and cokernel of the map $(c^*\phi)^i$, respectively. We also have the following exact sequence.
$$
0\longrightarrow \underline{\textsf{K}}^i_c\longrightarrow (\pi_r)_*(c^*\mathcal E)\xrightarrow{(\pi_r)_*((c^*\phi)^i)=((\pi_r)_*(c^*\phi))^i} (\pi_r)_*(c^*\mathcal E\otimes \omega^{\otimes i}_{c^*\mathfrak X})\longrightarrow \underline{\textsf{CK}}^i_c\longrightarrow 0
$$
Here $\underline{\textsf{K}}^i_c$ and $\underline{\textsf{CK}}^i_c$ are kernel and cokernel, respectively. Notice from Lemma \ref{pushforward1122} it follows that $\underline{\textsf{K}}^i_c=(\pi_s)_*(\textsf{K}^i_c)$. Since the $\Type~~(c^*{\mathfrak X}, c^*\mathcal E, c^*\phi)$ is constant over $C$, therefore the Hilbert polynomial of $\underline{\textsf{K}}^i_c$ does not depend on $c\in C$ which in turn implies that the Hilbert polynomial of $\underline{\textsf{CK}}^i_c$ does not depend on $c\in C$. Consider next the sheaf cokernel $\underline{\textsf{CK}}^i:=Coker((\pi_r)_*((\phi)^i)$. Since $\underline{\textsf{CK}}^i$ is a cokernel, we have that $c^*\underline{\textsf{CK}}^i=\underline{\textsf{CK}}^i_c$. Therefore, the sheaf $\underline{\textsf{CK}}^i$ is flat over $C$. Now consider the kernel sheaf $\underline{\textsf{K}}^i:=Ker((\pi_r)_*((\phi)^i))$. Since in the four terms sequence 
\[
0\longrightarrow \underline{\textsf{K}}^i \longrightarrow (\pi_r)_*(\mathcal E)\xrightarrow{(\pi_r)_*((\phi)^i)=((\pi_r)_*(\phi))^i} (\pi_r)_*(\mathcal E\otimes \omega^{\otimes i}_{\mathfrak X})\longrightarrow \underline{\textsf{CK}}^i\longrightarrow 0
\]
the last three terms are flat over $C$, we conclude that $\underline{\textsf{K}}^i$ will also be flat over $C$, and therefore $\textsf{K}^i$ is flat over $C$.
\end{proof}

\

\begin{Prop}\label{DefCom}
The tangent complex of $\mathcal Nilp^{red}_{0}$ at a point $(X_r,  E, \phi)\in \mathcal Nilp^{sm, gen}_0$ is given by $R\Gamma(\mathcal {SC}{( E, \phi)})$, where $\mathcal {SC}{(E, \phi)}$ is the following complex of sheaves on $X_r$.
\begin{equation}\label{above}
\mathcal {SC}{(E, \phi)}:=[SC(E,\phi)\xrightarrow{[-, \phi]} SC(E,\phi)\otimes \omega_{X_r}]
\end{equation}
where $SC(E, \phi)\subseteq \End~E$ is the sheaf of local sections $s$ of $\End~E$ such that $s(E^i)\subseteq E^{i-1}$ for $i=1,\dots, k+1$.
\end{Prop}
\begin{proof}
We choose a trivialisation $X_r=\cup_{i\in I} V_i$ of the vector bundle $\mathcal E$ and the line bundle $\omega_{X_r}$. Then a first-order infinitesimal deformation (as a Higgs bundle) of $( E, \phi)$ can be described as a pair $(s_{ij}, t_i)$, where $s_{ij}\in \Gamma(V_{ij}, \End E)$ and $t_i\in \Gamma(V_i, \End E\otimes \omega_{X_r})$ satisfying $s_{ij}+s_{jk}=s_{ik}$ and $t_i-t_j=[s_{ij}, \phi]$ (see \cite[Theorem 2.3]{BR} and \cite[Proposition 3.1.2]{Bo}).

Let $\Spec (\mathbb k[\epsilon]) \longrightarrow \mathcal Nilp^{sm, gen}_0$ be a map such that the image of the closed point is given by the Higgs bundle $(X_r, E, \phi)$. Let us denote by the corresponding first-order infinitesimal deformation of the nilpotent Higgs bundle $(E, \phi)$ by $(E[\epsilon], \phi[\epsilon])$. We assume that $(E[\epsilon], \phi[\epsilon])$ is a nilpotent Higgs bundle. We define $ E^i[\epsilon]:=\ker~(\phi[\epsilon])^i$ for $i\in [1, k+1]$. Since the induced flag $E^{\bullet}$ is flat over $\mathcal Nilp^{sm, gen}_0$ (see Lemma \ref{flat23}), therefore $ E^{\bullet}[\epsilon]$ is flat over $\Spec(\mathbb k[\epsilon])$. Since $E^{i}[\epsilon]$ is flat over $\Spec(\mathbb k[\epsilon])$, therefore it is an extension of $E^i$ by $E^i$.  

\begin{equation}
    \begin{tikzcd}
        0\rar & E \rar & E[\epsilon] \rar & E \rar & 0\\
        0\rar & E^i \rar \arrow[hook]{u} & E^i[\epsilon] \rar \arrow[hook]{u} & E^i \rar \arrow[hook]{u} & 0
    \end{tikzcd}
\end{equation}
\vspace{.6em}

Now it is straightforward to check that for $(s_{ij}, t_i)$ to be an infinitesimal deformation of $(E, \phi)$ as a nilpotent Higgs field, it has to satisfy the extra condition that  $s_{ij}(E^{\bullet})\subseteq E^{\bullet-1}$ and $t_{i}(E^{\bullet})\subseteq E^{\bullet-1}\otimes \omega_{X_r}$, where $ E^{\bullet}$ is the flag \eqref{fflag}. This means that $s_{ij}\in \Gamma(V_{ij}, SC(E, \phi))$ for all $i,j$ and therefore, the proposition follows.
\end{proof}

\

\begin{Thm}\label{isotropic112}
\begin{enumerate}
\item The Hitchin map $h: M^{cl}_{Gie}\longrightarrow B$ is surjective.
\item The sub-stack $\mathcal Nilp^{sm, gen}$ is relatively isotropic in $M^{cl}_{Gie}$ with respect to the relative 0-shifted log symplectic form defined in Theorem \ref{main112}.
\item The Hitchin map $h: M^{cl}_{Gie}\longrightarrow B$ is flat.
\end{enumerate}
\end{Thm}

\begin{proof}

\text{\underline{Proof of (1):}} Let $a_{\bullet}:=(a_1,\dots, a_n)\in B$ be any point. This point $a_{\bullet}$ can either lie over the closed fibre of the map $B\to S$ or outside the closed fibre. If it lies outside the closed fibre, then it lies in the image of the Hitchin map. This fact follows from the spectral correspondence (\cite[Theorem 6.11]{32.1}). If $a_{\bullet}$ lies over the closed fibre, we again use a version of spectral correspondence for the Hitchin map for nodal curves \cite[Lemma 2.4]{2a}. Let us consider the function $s(a_{\bullet}):=t^n+a_1t^{n-1}+\dots+a_{n-1}t+a_n$ on the total space $\Tot(\omega_{X_0})$. Here $t$ denotes the canonical section of $f^*\omega_{X_0}$, where $f: \Tot(\omega_{X_0})\longrightarrow X_0$ is the projection map. Then the vanishing locus $V(s(a_{\bullet}))$ defines a closed sub-scheme of $\Tot(\omega_{X_0})$. Notice that $\Tot(\omega_{X_0})$ is only quasi-projective and it is an open subscheme of $Z:=\mathbb P(\omega_{X_0}^*\oplus \mathcal O_{X_0})$. But since $s(a_{\bullet})$ is a monic polynomial, therefore the closure of $V(s(a_{\bullet}))$ in $Z$ is $V(s(a_{\bullet}))$ itself. In particular, $V(s(a_{\bullet}))$ is a closed subscheme in $Z$ such that $V(s(a_{\bullet}))\cap D_{\infty}=\emptyset$, where $D_{\infty}:=Z\setminus \Tot(\omega_{X_0})$. Therefore, by spectral correspondence, it follows that any rank $1$ locally free sheaf on $V(s(a_{\bullet}))$ corresponds to a Higgs bundle $(E, \phi)$ on the nodal curve $X_0$ whose characteristic polynomial is given by $a_{\bullet}:=(a_1,\dots, a_n)\in B$. Therefore, the Hitchin map is surjective.

\vspace{1em}

\text{\underline{Proof of (2):}} We want to show that $\mathcal Nilp^{sm,gen}$ is relatively isotropic over $S$. First, we notice that the smooth locus of the generic fibre of $\mathcal Nilp^{red}\longrightarrow S$ is isotropic, and its dimension is equal to $n^2(g-1)$, which is the same as the dimension of the stack of rank $n$ vector bundles on a smooth projective curve of genus $g$. Therefore, we have $ \dim \mathcal Nilp^{red}_0\geq n^2(g-1)$. We want to show that the $\mathcal Nilp^{red,sm}_0$ is an isotropic substack in $M^{cl}_{Gie,0}$. The tangent complex at a point $(X_r, E, \phi)$ of $\mathcal Nilp^{red, sm}_0$ is given by the derived global sections of \eqref{above}. The tangent complex at the corresponding point of $M^{cl}_{Gie}$ is given by the derived global sections of the following complex (see \eqref{TC113} and Theorem \ref{TC112}). 

\begin{equation}
\mathcal {C}{(E, \phi)}:=[\End E\xrightarrow{[-, \phi]} \End E\otimes \omega_{X_r}]
\end{equation}

We want to show that the following composite morphism is homotopic to $0$. 
\begin{small}
\begin{equation}\label{isotropic11}
\begin{tikzcd}
    R\Gamma(\mathcal {SC}(E, \phi))[1]\arrow{r} \arrow[bend right=20]{rrr} & R\Gamma(\mathcal {C}(E, \phi))[1]\arrow{r}{\omega^{\flat}} & (R\Gamma(\mathcal {C}( E, \phi))[1])^{\vee}\arrow{r} & (R\Gamma(\mathcal {SC}( E, \phi))[1])^{\vee}
\end{tikzcd} 
\end{equation}
\end{small}
Notice that the morphism $\omega^{\flat}: R\Gamma(\mathcal {C}(E, \phi))[1]\to (R\Gamma(\mathcal {C}(E, \phi))[1])^{\vee}$ is given by the logarithmic-symplectic pairing (Theorem \ref{main112}). If we show that the composite morphism \eqref{isotropic11} is homotopic to $0$, this will imply that $\mathcal Nilp^{red, sm}_0$ is an isotropic substack. To show this, we choose a cover $X_r=\underset{i\in I}\cup V_i$ for which the vector bundle $E$ and the line bundle $\omega_{X_r}$ trivialize. With respect to this cover, the complex $R\Gamma(\mathcal {SC}(E, \phi))[1]$ is quasi-isomorphic to the following \v{C}ech complex
\begin{equation}
 \resizebox{1.00\hsize}{!}{$\prod SC( E,\phi)(V_i)\longrightarrow \prod SC( E,\phi)(V_{ij})\oplus \prod(SC(E,\phi)\otimes \omega_{X_r} )({V_i}) \longrightarrow \prod(SC( E,\phi)\otimes \omega_{X_r} )(V_{ij})\oplus \prod SC(E,\phi)(V_{ijk})$,}
\end{equation}
where the term $\prod SC(E,\phi)({V_i})$ is in degree $-1$. 

Notice that the composite map can only be non-zero on three cohomologies: in the degree $-1$, $0$ and $1$. 

\

\underline{Case 1: the induced map $H^0(\omega^{\flat})$:} An element of \[\prod SC(E,\phi)(V_{ij})\oplus \prod(SC(E,\phi)\otimes \omega_{X_r} )({V_i})\] is given by $\{(s_{ij}, t_i)\}_{i,j\in I}$, where $(s_{ij}, t_i)$, where $s_{ij}\in \Gamma(V_{ij}, \End E)$ and $t_i\in \Gamma(V_i, \End E\otimes \omega_{X_r})$ such that $s_{ij}(E^{\bullet})\subseteq E^{\bullet-1}$, and $t_{i}(E^{\bullet})\subseteq E^{\bullet-1}\otimes \omega_{X_r}$, where $E^{\bullet}$ is the flag \eqref{fflag}. Then from \eqref{isotropic11} it follows that the pairing of two such elements $\{(s_{ij}, t_i)\}$ and $\{(s'_{ij}, t'_i)\}$ is given by \[\{T_{ij}:=Trace(s_{ij}\circ t'_j-t_i\circ s'_{ij})\in \omega_{X_r}(V_{ij})\}_{i,j\in I}.\] 

\

\underline{Claim:} The section $T_{ij}$ vanishes for all $i,j\in I$.

\

\underline{Proof of the claim:} We first observe that, we can choose the trivialisation $X_r=\cup_{i\in I} V_i$ of $E$ and $\omega_{X_r}$ in such a way that 
\begin{enumerate}
\item $V_i$ contains at most one node for each $i\in I$, 
\item $V_i$'s are connected.
\end{enumerate}

We notice that $V_{ij}$'s also contain at most one node. Suppose, $V_{ij}$ contains a node $p$. Then $V_{ij}=V^1_{ij}\coprod V^2_{ij}$, the union of two smooth irreducible components of $V_{ij}$. We denote by $V^o_{ij}:=V_{ij}\setminus p$, $V^{1,o}_{ij}:=V^1_{ij}\setminus p$ and $V^{2,o}_{ij}:=V^2_{ij}\setminus p$. We consider the restriction of the section $T_{ij}$ to the two open subsets $V^{1,o}_{ij}$ and $V^{2,o}_{ij}$. Notice that the restrictions of the flag \eqref{fflag} to these two open subsets are all sub-bundles. We notice that $SC(E|_{V^{1,o}_{ij}},\phi)$ is the nilpotent part of the parabolic sub-algebra of $\End (E|_{V^{1,o}_{ij}})$ given by the natural flag of sub-bundles induced by $\phi$. Therefore it is clear that the trace pairing is $0$ i.e., $(T_{ij})|_{V^{1,o}_{ij}}=0$. Similarly, $(T_{ij})|_{V^{2,o}_{ij}}=0$. Since $V^o_{ij}$ is a dense open subset of $V_{ij}$, therefore, by continuity, $T_{ij}=0$. 

If a $V_{ij}$ does not contain any node, then the proof is similar. 

\

\underline{Case 2: the induced map $H^{-1}(\omega^{\flat})$:} In this case, we look at the pairing of $\{t_i\in SC(E, \phi)\otimes \omega_{X_r}\}$ and $\{s_{ij}\in SC(E, \phi)\}$. The pairing is given by $\{\tr(t_i\circ s_{ij})\}$. By similar arguments as above, we can show that $\tr(t_i\circ s_{ij})=0~~\forall i, j\in I$.

\

\underline{Case 3: the induced map $H^1(\omega^{\flat})$:} In this case, we look at the pairing of $\{t_{ij}\in (SC(E, \phi)\otimes \omega_{X_r})(V_{ij})\}$ and $\{s_j\in SC(E, \phi)(V_j)\}$. The pairing is given by $\{\tr(t_{ij}\circ s_j)\}$. By similar arguments as above, we can show that $\tr(t_{ij}\circ s_j)=0~~\forall i,j\in I$.

\

Therefore, we conclude that the smooth locus of this particular component is isotropic.

\

\text{\underline{Proof of (3)}} From (2), it follows that $ \dim \mathcal Nilp^{red}_0= n^2(g-1)$. Now from Lemma \ref{dim111}, it follows that the dimension of $M^{cl}_{Gie}$ is equal to $2n^2(g-1)+2$. Once we know the dimension, it is easy to see that $M^{cl}_{Gie}$ is a local complete intersection (Corollary \ref{LCI}). Since $M^{cl}_{Gie}$ is a local complete intersection, and all the fibres of the Hitchin map $h: M^{cl}_{Gie}\longrightarrow B$ are of dimension equal to $n^2(g-1)$, therefore by the miraculous flatness criterion \cite[Exercise 10.9, page 276]{Ha}, it follows that the Hitchin map is flat.
\end{proof}

\

\section{\textbf{On the relative logarithmic Dolbeault moduli over \texorpdfstring{$\overline {\mathcal M_g}$}{Mgbar}}}
In this section, we will construct the derived moduli stack $M_{g}^{Dol}$ of Higgs bundles with coefficients in the relative dualizing sheaf along the fibers of the universal curve over the moduli stack $\mathcal{M}_{g}^{ss}$ of semistable curves of genus $g\geq 2$.  We will show that there is a relative $0$-shifted symplectic form on $M_{g}^{Dol}/\mathcal{M}_{g}^{ss}$ which can also be interpreted as a relative $0$-shifted log-symplectic form on $M_{g}^{Dol}$, relative to the moduli stack of stable curves $\overline{\mathcal{M}_{g}}$. Let us first recall some notations.

\

\noindent
A connected projective variety $C$ of dim $1$ is called a \emph{\bfseries stable curve} if it is either smooth or has nodal singularities, and the automorphism group $\Aut (C)$ is finite. 

The \emph{\bfseries moduli stack of stable curves} represents the moduli problem
$$
\overline{\mathcal M_g}: (Sch/\mathbb k)^{op} \longrightarrow Groupoids 
$$ 
\begin{equation}
T
\mapsto \left\{
\begin{array}{@{}ll@{}}
\text{families of stable curves}\\
\text{ of genus} ~g~\text{over}~T
\end{array}\right\}
\end{equation}

\vspace{1em}

Here, morphisms in the groupoid $\overline{\mathcal M_g}(T)$ are isomorphisms of stable curves over $T$.

\ 

\begin{Rem}\label{Rem:Moch} It is well known that if $g\geq 2$, then $\overline{\mathcal M_g}$ is a smooth and proper Deligne-Mumford stack. The locus of nodal curves forms a normal crossing divisor in $\overline{\mathcal M_g}$ and so gives a divisorial log structure on $\overline{\mathcal M_g}$. On the other hand, by a well known result of S. Mochizuki \cite[\S 3B.]{M}, given any family of semi-stable curves $\mathcal C\to T$, there is a natural log structure on $\mathcal C$ and also on the base scheme $T$ such that the projection map is a log-smooth map and the relative logarithmic cotangent bundle is the relative dualizing sheaf of $\mathcal C/T$. We call these log structures the \emph{\bfseries basic log structure package} on $\mathcal C/T$. It is known that the log structure on $\overline{\mathcal M_g}$ which is part of the basic log structure package for the universal curve $\mathcal C_g$ over $\overline{\mathcal M_g}$ agrees with the divisorial log-structure on $\overline{\mathcal M_g}$ (see the comment after Lemma 4.4 in \cite{17} and the proof of our Proposition \ref{logstruc111}).
\end{Rem}

\

\noindent
A connected projective variety $C$ of dim $1$ is called a \emph{\bfseries semi-stable curve} if the following properties are satisfied
\begin{enumerate} 
\item it is either smooth or has nodal singularities, 
\item every rational irreducible component $R$ is smooth, and 
\item $R\cdot \overline{(C\setminus R)}=2$.
\end{enumerate}

\medskip

\noindent
The \emph{\bfseries moduli stack of semi-stable curves} represents the moduli problem
$$
{\mathcal M^{ss}_g}: Sch/\mathbb k\longrightarrow Groupoids $$ 
\begin{equation}
T
\mapsto \left\{
\begin{array}{@{}ll@{}}
\text{families of semi-stable curves}\\
\text{of genus} ~g~\text{over}~T
\end{array}\right\}
\end{equation}

\vspace{1em}

Morphisms in the groupoid ${\mathcal M^{ss}_g}(T)$ are again isomorphisms of families of semi-stable curves over $T$.

\medskip

\begin{Rem} It is well known that if $g\geq 2$,  ${\mathcal M^{ss}_g}$ is a smooth Artin stack \cite[0E72, Lemma 21.5]{34} of finite type.
\end{Rem}

\subsection{\textbf{Log structures on \texorpdfstring{${\mathcal M^{ss}_g}$}{Mssg}}} The locus of singular curves in ${\mathcal M^{ss}_g}$ again forms a normal crossing divisor, and hence induces a log structure on ${\mathcal M^{ss}_g}$. We denote this divisor by $\partial \mathcal M^{ss}_g$. Let us denote the universal curve over ${\mathcal M^{ss}_g}$ by $\mathcal D_g$. There is the ``stabilization morphism" $\pi: \mathcal M^{ss}_g\longrightarrow \overline{\mathcal M_g}$ such that the induced morphism $\mathcal D_g\longrightarrow \pi^*\mathcal C_g$ is the universal modification morphism over $\mathcal M^{ss}_g$. The existence of the universal modification follows from a well-known result called the stable reduction theorem \cite[Lemma 109.23.4.]{34}. The existence of the stabilisation morphism follows from the universal property of $\overline{\mathcal M_g}$ and the universal modification morphism. 

\subsubsection{\textbf{Versal deformations of the stabalization morphism}}\label{VerPic}

Let $\mathcal X$ be a stable curve of genus $g$ over $\spec \mathbb k$ and $\mathfrak X$ be a semi-stable curve whose stable model is $\mathcal X$. Let $\{c_i\}^l_{i=1}$ be the nodes of the curve $\mathcal X$. Let $\{d_{ij}\}^{\iota_i}_{i=1}$ be the nodes of $\mathfrak X$ over the node $c_i$ for every $i=1, \dots, l$. From \cite[Proposition 3.3.2]{29}, it follows that 
\begin{enumerate}
\item there exists a versal deformation space of the nodal curve $\mathcal X$ which is isomorphic to $\spec \mathbb k[|z_1, \dots, z_l|][|z_{l+1}, \dots, z_N|]$, where $N:=3g-3$ (here $g$ is the genus of the curves) and $z_i$ is the equation of the $i$-th node of $\mathcal X$ for $i=1, \dots, l$, and 
\item there exits a versal deformation space of the nodal curve $\mathfrak X$ which is isomorphic to $\spec \mathbb k[|\{\{z_{ij}\}^{l}_{i=1}\}^{\iota_i}_{j=1}|][|z_{l+1}, \dots, z_N|]$, where $z_{ij}=0$ is the equation of the node $d_{ij}$ for $i=1, \dots, l$ and $j=1, \dots, \iota_i$. Moreover, there is a morphism between the versal deformation spaces, which is given as follows. 
\begin{align*}
    \mathbb k[|z_1, \dots, z_l|][|z_{l+1}, \dots, z_N|]\longrightarrow \mathbb k[|\{\{z_{ij}\}^{l}_{i=1}\}^{\iota_i}_{j=1}|][|z_{l+1}, \dots, z_N|]\\
    z_i\mapsto z_{i1}\cdots z_{i\iota_i}\hspace{3em}\forall i=1,\dots, l,~~~~\text{and}\\
   z_{j}\mapsto z_j\hspace{3em}\forall j=l+1,\dots, N
\end{align*}

This morphism is the local analytic picture of the stabilisation morphism $\pi: \mathcal M^{ss}_g\longrightarrow \overline{\mathcal M_g}$ at a given point of $\mathcal M^{ss}_g$. 
\end{enumerate}

\begin{Prop}\label{logstruc111}
\begin{enumerate}
\item The log structures on $\overline{\mathcal M_g}$ and $\mathcal M^{ss}_g$ which are parts of the basic log structure package on $\mathcal C_g /\overline{\mathcal M_g} $ and $\mathcal D_g / \mathcal M^{ss}_g$ (see Remark \ref{Rem:Moch}) are locally free. These log structures coincide with the divisorial log structures induced by the boundary divisors $\partial \overline{\mathcal M_g}$ and $\partial \mathcal M^{ss}_g$, respectively. 
\item The morphism $\mathcal M^{ss}_g\longrightarrow \overline{\mathcal M_g}$ is a log-smooth morphism with respect to the log-structures mentioned in (1).
\end{enumerate}
\end{Prop}
\begin{proof}
\underline{Proof of (1):}
The fact that the log structure from the basic log structure packages and the log structures induced from the boundary divisors are the same can be seen as follows.

Given a point $p\in \overline{\mathcal M_g}$ and the set of nodes $\{q_1, \dots, q_n\}\in \mathcal C_{g, p}$ (the fiber over $p$ of the universal curve $\mathcal C_{g}$), we have a $\mathbb{k}$ algebra $B$ and an element $b_i$ fitting into a commutative diagram of \'{e}tale neighbourhoods of the points $p$ and $q_i \in \{q_1, \dots, q_n\}$ (using Artin approximation):

\begin{equation}
\begin{tikzcd}
\mathcal C_g\arrow{d} && U_i \arrow{ll}{etale}\arrow{drr}\arrow{rr}{etale} && \spec (B[x,y]/(xy-b_i))\arrow{d} \arrow{r} &\spec (\mathbb k[x,y,a]/(xy-a)) \arrow{d}\\
\overline{\mathcal M_g} & & & & \spec B \arrow{llll}{etale} \arrow{r} & \spec \mathbb k[a] 
\end{tikzcd}
\end{equation}

Here the rightmost square is cartesian and $U_i$ denotes a suitable \'{e}tale neighbourhood of $q_i$ and not containing any other nodes in $\{q_1, \dots, q_n\}$. The scheme $\spec \mathbb k[a]$ is a versal deformation space of a node (i.e., $\spec \mathbb k[x,y]/(x y$)) and the family \[\spec (\mathbb k[x,y,a]/(xy-a))\to \spec \mathbb k[a]\] is a versal deformation family of curves for a node (see \cite[Theorem 14.1]{Hart}). If we consider the basic log structure package on $\spec (\mathbb k[x,y,a]/(xy-a)) / \Spec~(\mathbb k[a])$, the basic log structures on $\Spec~\mathbb k[a]$, agrees by Kato \cite[Lemma 4.4]{17} with the divisorial log structure given by the boundary divisor (i.e., the vanishing locus of the element $a\in \mathbb k[a]$). Now from \cite[Lemma 2.1, Lemma 2.2, Proposition 2.3]{17}, we can find an \'{e}tale neighbourhood $V$ of the point $p\in \overline{\mathcal M_g}$, and define a log structure by taking the amalgamated sum of the log structures contributed by each of the nodes in $\{q_1, \dots, q_n\}$.  But this is the same as the log structures induced by the pullback of the boundary divisor in $\overline{\mathcal M_g}$. This completes the argument. Similarly, the two log structures on $\mathcal M^{ss}_g$ are the same. 

\underline{Proof of (2):} Again using Artin approximation we have a commutative diagram 

\begin{equation}
\begin{tikzcd}
\mathcal M^{ss}_g\arrow{d} & U \arrow{l}{etale}\arrow{d} & V \arrow{r}{etale}\arrow{l}{etale} & \Spec ~\mathbb k[\{\{z_{ij}\}^{l}_{i=1}\}^{\iota_i}_{j=1}][\{z_{k}\}^N_{k=l+1}] \arrow{d}\\
\overline{\mathcal M_g} & W\arrow{l}{etale} \arrow{rr}{etale} &&  \Spec \mathbb k[z_1, \dots, z_l][z_{l+1}, \dots, z_N] 
\end{tikzcd}
\end{equation}
where all the squares and triangles are commutative and only the left square is Cartesian. 

It can be rewritten as the following commutative diagram
\begin{equation}
\begin{tikzcd}
\mathcal M^{ss}_g\arrow{d} & V \arrow{l}{etale}\arrow{d} \arrow{rr}{etale} &  & \Spec ~\mathbb k[\{\{z_{ij}\}^{l}_{i=1}\}^{\iota_i}_{j=1}][\{z_{k}\}^N_{k=l+1}] \arrow{d}\\
\overline{\mathcal M_g} & W\arrow{l}{etale} \arrow{rr}{etale} && \Spec  \mathbb k[z_1, \dots, z_l][z_{l+1}, \dots, z_N] 
\end{tikzcd}
\end{equation}

Notice that the rightmost vertical morphism is a morphism of log schemes. Since the right commutative diagram is commutative and the log structures on $V$ and $W$ are isomorphic to the pullback of the log structures from $\Spec ~\mathbb k[\{\{z_{ij}\}^{l}_{i=1}\}^{\iota_i}_{j=1}][\{z_{k}\}^N_{k=l+1}]$ and $\mathbb k[z_1, \dots, z_l][z_{l+1}, \dots, z_N]$, therefore we have the structure of a log-morphism on $V\rightarrow W$.  

Now we cover $\mathcal M^{ss}_g$ and $\overline{\mathcal M_g}$ with compatible \'etale open subsets with log morphisms between them. Now we need to show that these log-morphisms glue i.e., they are equal on the intersections. But this can be checked by passing to the completion, where the log structure is given by the following chart:

\begin{equation}
\begin{tikzcd}
& e_i\arrow{rr} && \sum^{\iota_i}_{j=1} e_{ij}\\
& \mathbb N^{\oplus l}\arrow{d}\arrow{rr} && \oplus^l_{i=1}(\mathbb N^{\oplus \iota_i})\arrow{d} & 
\\
& \mathbb k[|z_1, \dots, z_l|][|z_{l+1}, \dots, z_N|] \arrow["="]{rr}[swap]{z_i=z_{i1}\cdots z_{i\iota_i} } && \mathbb k[|\{\{z_{ij}\}^{l}_{i=1}\}^{\iota_i}_{j=1}|][|\{z_{k}\}^N_{k=l+1}|] & 
\end{tikzcd}
\end{equation}
where $e_i$ is sent to $z_i$ and $e_{ij}$ is sent to $z_{ij}$, where $i=1,\dots , l$.

The above map can be seen a tensor products of the following maps over the field $k$ along with an extra tensor product with $k[|z_{l+1}, \dots, z_N|]$, where  $k[|z_{l+1}, \dots, z_N|]$ does not contribute at all in the log structure. 

\begin{equation}
\begin{tikzcd}
& e_i\arrow{rr} && \sum^{\iota_i}_{j=1} e_{ij}\\
& \mathbb N\arrow{d}\arrow{rr} && \mathbb N^{\oplus \iota_i}\arrow{d} & 
\\
& \mathbb k[|z_i|] \arrow["="]{rr}[swap]{z_i=z_{i1}\cdots z_{i\iota_i} } && \mathbb k[|\{z_{ij}\}^{\iota_i}_{j=1}|] & 
\end{tikzcd}
\end{equation}
where $e_i$ is sent to $z_i$ and $e_{ij}$ is sent to $z_{ij}$.

Now it is straightforward to check that each of these maps is log-smooth using Kato's criterion (c.f. \cite[Theorem 3.5]{17b}). Since the product of log-smooth maps is log-smooth, it follows that the map $\mathcal M^{ss}_g\longrightarrow \overline{\mathcal M_g}$ is log-smooth. 
\end{proof}

\subsection{\textbf{Relative log-cotangent complex of the map \texorpdfstring{${\mathcal M^{ss}_g}\stackrel{f}\longrightarrow \overline{\mathcal M_g}$}{MssgtoMg}}}

\begin{Prop}\label{reltan}
The relative logarithmic cotangent complex of $f$ is trivial: $\mathbb L^{\log}_{f}\cong 0$. 
\end{Prop}
\begin{proof}
We have a commutative diagram (not a cartesian square)

\begin{equation}
\begin{tikzcd}
\mathcal D_g\arrow{d}{\tilde p} \arrow{rr}{\tilde f} & & \mathcal C_g\arrow{d}{p}\\
\mathcal M^{ss}_g\arrow{rr}{f} && \overline{\mathcal M_g}  
\end{tikzcd}
\end{equation}
\vspace{1em}

These families of semi-stable curves as discussed in Remark \ref{Rem:Moch} have relative log structure packages, defining log structures on all the spaces in this diagram, for which the projection morphisms $p$ and $\tilde{p}$ are log-smooth (\cite[Theorem 2.1., and ``Gobal Construction", Page 227]{17}). The log-tangent complex of $\overline{\mathcal M_g}$ is isomorphic to $Rp_*(\omega^{\vee}_{\mathcal C_g/ \overline{\mathcal M_g}})[1]$ and the log-tangent complex of $\mathcal M^{ss}_g$ is isomorphic to $R\tilde p_*(\omega^{\vee}_{\mathcal D_g/ {\mathcal M^{ss}_g}})[1]$. Now notice that $\tilde{f}^{*}\omega^{\vee}_{\mathcal C_g/ \overline{\mathcal M_g}} \cong \omega^{\vee}_{\mathcal D_g/ {\mathcal M^{ss}_g}}$ (see \cite[Page 169]{KFF}) and that the fibers of $\tilde{f}$ are connected and so $R\tilde{f}_{*}\mathcal{O} \cong \mathcal{O}$. Hence, $R\tilde{f}_{*}\omega^{\vee}_{\mathcal C_g/ \overline{\mathcal M_g}} \cong \omega^{\vee}_{\mathcal D_g/ {\mathcal M^{ss}_g}}$. Therefore we have, 

\[Rf_*\circ R\tilde p_*(\omega^{\vee}_{\mathcal D_g/ {\mathcal M^{ss}_g}})[1]=Rp_*R\tilde f_*(\omega^{\vee}_{\mathcal D_g/ {\mathcal M^{ss}_g}})[1]\cong Rp_*(\omega^{\vee}_{\mathcal C_g/ \overline{\mathcal M_g}})[1].\] 

Therefore, we conclude that the natural map between the two log-tangent complexes ( and the map between log-cotangent complexes) is an isomorphism. We have the following distinguished triangles of log-cotangent complexes
\begin{equation}
f^*\mathbb L^{\log}_{\overline {\mathcal M_g}}\longrightarrow \mathbb L^{\log}_{{\mathcal M^{ss}_g}}\longrightarrow \mathbb L^{\log}_f\longrightarrow f^*\mathbb L^{\log}_{\overline {\mathcal M_g}}[1] 
\end{equation}
of the map \[f: \mathcal M^{ss}_g\to \overline{\mathcal M_g}.\]

Since, the map $f^*\mathbb L^{\log}_{\overline {\mathcal M_g}}\longrightarrow \mathbb L^{\log}_{{\mathcal M^{ss}_g}}$ is an equivalence, therefore we conclude that $\mathbb L^{\log}_f\cong 0$. 
\end{proof}

\subsection{\textbf{Relative logarithmic Dolbeault shape and shifted symplectic forms}}

Let $\mathcal D^{Dol}_g$ denote the relative logarithmic Dolbeault moduli stack for the family of curves $\mathcal D_g\longrightarrow \mathcal M^{ss}_g$. Then $\mathcal M^{Dol}_g:=\map_{\mathcal M^{ss}_g}(\mathcal D^{Dol}_g, BGL_n\times \mathcal M^{ss}_g)$ is the relative derived moduli stack of Gieseker-Higgs bundles viewed over $\overline{\mathcal M_g}$. 

\begin{Prop}
The morphism $\mathcal M^{Dol}_g\longrightarrow \mathcal M^{ss}_g$ is a quasi-smooth morphism of derived Artin stacks.
\end{Prop}
\begin{proof}
Same as the proof of Proposition \ref{qs11}.
\end{proof}

We equip $\mathcal M^{Dol}_g$ with the locally free log-structure pulled back from $\mathcal M^{ss}_g$ via the morphism $\mathcal M^{Dol}_g\longrightarrow \mathcal M^{ss}_g$.

\begin{Thm}\label{main1567}
There is a $0$-shifted relative log-symplectic form on $\mathcal M^{Dol}_g$ (relative to the moduli stack of stable curves $\overline{\mathcal M_g}$).
\end{Thm}
\begin{proof}
Consider the composite morphism $\mathcal M^{Dol}_g\xrightarrow{\pi} \mathcal M^{ss}_g\xrightarrow{f} \overline{\mathcal M_g}$. We have a distinguished triangle
\begin{equation}
\pi^*\mathbb L^{\log}_{\mathcal M^{ss}_g/\overline{\mathcal M_g}}\longrightarrow \mathbb L^{\log}_{\mathcal M^{Dol}_g/\overline{\mathcal M_g}}\longrightarrow \mathbb L^{\log}_{\mathcal M^{Dol}_g/\overline{\mathcal M_g}}\longrightarrow \pi^*\mathbb L^{\log}_{\mathcal M^{ss}_g/\overline{\mathcal M_g}}[1].
\end{equation}

But since $\mathbb L^{\log}_{\mathcal M^{ss}_g/\overline{\mathcal M_g}}\cong 0$ (Proposition \ref{reltan}), therefore the relative log cotangent complex of the composite morphism is isomorphic to the relative log cotangent complex of the morphism $\pi$. But notice that the log structure of $\mathcal M^{Dol}_g$ is pulled back of the log structure of $\mathcal M^{ss}_g$ via the map $\pi$. Therefore the relative log-cotangent complex of the morphism $\pi$ is isomorphic to the relative cotangent complex of the morphism $\pi$. Now since the relative logarithmic Dolbeault stack over the moduli stacks $\mathcal M^{Dol}_g$ is $\mathcal O$-compact and $\mathcal O$-oriented (Theorem \ref{main}); therefore $\mathcal M^{Dol}_g$ has a $0$-shifted relative symplectic form over $\mathcal M^{ss}_g$, which is a $0$-shifted relative log-symplectic form viewed over the moduli stack of stable curves $\overline{\mathcal M_g}$.
\end{proof}

\

\
\appendix
\section{Classical Artin stack of Gieseker-Higgs bundles and its local properties}

In this appendix we study the classical truncation $M^{cl}_{Gie}$ of the derived moduli stack of Higgs bundles (Section 3), restricted to the Gieseker locus (Definition \ref{gis1}). Our main results are:

- The stack of Gieseker vector bundles $N_{Gie}$ (and in particular its closed fiber $N^{cl}_{Gie,0}$) is an almost very good stack in the sense of Soibelman \cite[Definition 2.1.2]{32e}.

- The classical stack of Gieseker-Higgs bundles over the special fiber, $M^{cl}_{Gie,0}$, is an irreducible local complete intersection of pure dimension $2\dim N^{cl}_{Gie,0} + 1$.

These facts are used in Section 5 to deduce flatness of the Hitchin map on the Gieseker locus.

\subsection{The classical Artin stacks of Gieseker vector bundles and Gieseker Higgs bundles}\label{Gies21}

We now describe the underlying classical Artin stacks associated to the derived moduli problems considered earlier. Let $N_{Gie}$ denote the classical Artin stack parameterizing Gieseker vector bundles, and let $M^{cl}_{Gie}$ denote the classical Artin stack parameterizing Gieseker–Higgs bundles (both defined over the base $S$).

\subsubsection{Gieseker vector bundles}
\
\begin{Def}\label{Gie1101}
A vector bundle $\mathcal E$ of rank $n$ on $X_r$ with $r\geq 1$ is called a \textbf{Gieseker vector bundle} if
 \begin{enumerate}
\item $\mathcal E|_{R[r]}$ is a strictly standard vector bundle on $R[r]\subset X_r$, i.e., for each $i=1,\dots, r $, $\exists$ non-negative integers $a_i$ and $b_i$ such that $\mathcal E|_{R[r]_i}\cong \mathcal O^{\oplus a_i}\oplus \mathcal O(1)^{\oplus b_i}$, and
\item the direct image $\pi_{r *}(\mathcal{E})$ is a torsion-free $\mathcal O_{X_0}$-module.
\end{enumerate}
\
Any vector bundle on $X_0$ is called a Gieseker vector bundle. In the literature, a Gieseker vector bundle is also called an admissible vector bundle.
\
A Gieseker vector bundle $(X_r, \mathcal E)$ is called a \textbf{stable Gieseker vector bundle} if $\pi_{r *}\mathcal{E}$ is a stable torsion-free sheaf on the irreducible nodal curve $X_0$, where $\pi_r: X_r\longrightarrow X_0$ is the natural contraction map.
\
A \textbf{(stable) Gieseker vector bundle} on a modification $\mathfrak X_T$ (see Definition \ref{gis120}) is a vector bundle such that its restriction to each $(\mathfrak X_{T})_t$ is a (stable) Gieseker vector bundle.
\end{Def}
\
\begin{Rem}\label{AltGie}
A Gieseker vector bundle can also be defined as a vector bundle $(X_r, \mathcal E)$ on a Gieseker curve $X_r$ satisfying the following two conditions
\begin{enumerate}
\item $\mathcal E|_{R[r]}$ is a globally generated vector bundle on $R[r]$,
\item the direct image $\pi_{r *}(\mathcal E)$ is a torsion-free $\mathcal O_{X_0}$-module.\end{enumerate}
\end{Rem}
\
\begin{Def}\label{gis1} A \textbf{Gieseker-Higgs bundle} on $\mathfrak X_T$ is a pair
$(\mathcal E_T , \phi_T )$, where $\mathcal E_T$ is a vector bundle on $\mathfrak X_T$, and $\phi_T : \mathcal E_T \longrightarrow \mathcal E_T \otimes \omega_{\mathfrak X_T/T}$ is an $\mathcal O_{\mathfrak X_T}$ -module homomorphism satisfying the following
\begin{enumerate}
\item $\mathcal E_T$ is a Gieseker vector bundle on $\mathfrak X_T$,
\item for each closed point $t \in T$ over $\eta_0 \in S$, the direct image $(\pi_t)_*(\mathcal E_t)$ is a torsion-free sheaf on $X_0$ and $(\pi_t)_*\phi_t: (\pi_t)_*(\mathcal E_t)\longrightarrow (\pi_t)_*(\mathcal E_t)\otimes \omega_{X_0}$ is an $\mathcal O_{X_0}$-module homomorphism. We refer to such a pair $((\pi_t)_*(\mathcal E_t), (\pi_t)_*\phi_t)$ as a torsion-free Higgs pair on the nodal curve $X_0$.
\end{enumerate}
\end{Def}
\
\begin{Rem}\label{AltGie2}
A Gieseker-Higgs bundle can also be defined as a Higgs bundle $(X_r, \mathcal E, \phi)$ on a Gieseker curve $X_r$ satisfying the following two conditions
\begin{enumerate}
\item $\mathcal E|_{R[r]}$ is a globally generated vector bundle on $R[r]$,
\item the direct image $\pi_{r*}(\mathcal E)$ is a torsion-free $\mathcal O_{X_0}$-module.
\end{enumerate}
\end{Rem}
\
\begin{Rem}
From \cite[Definition-Notation 1, Lemma 2, and Proposition 5]{25}, it follows that for the moduli problem of vector bundles (Higgs bundles) of rank $n$, we have to consider Gieseker curves $X_r$, where $r=0,\dots, n$.
\end{Rem}
We fix the rank and degree to be $n$ and $d$, respectively. Also, $n\geq 1$ and $d\in \mathbb Z$.
\
\subsubsection{\textbf{Stack of torsion-free Hitchin pairs}}\label{at3}
We recall the definition of \textbf{the moduli stack of torsion-free Hitchin pairs}
$$
\TFH(\mathcal X/S): {Sch}/S\rightarrow {Groupoids} $$
\begin{equation}
T
\mapsto \left\{
\begin{array}{@{}ll@{}}
\text{Families of torsion-free Hitchin pairs} \\
(\mathcal F_T,\phi_T: \mathcal F_T\longrightarrow \mathcal F_T\otimes \omega_{\mathcal X_T/T})\\
\text{of rank} ~n~ \text{and degree}~ d~ \text{over the original}\\
\text{family of curves}~~ \mathcal X_T/T~~ (\ref{DegeCourbes})
\end{array}\right\}
\end{equation}
It is an Artin stack. For the construction of an atlas for $\TFH(\mathcal X/S)$, we refer to \cite[5.0.7. The total family construction.]{2a}. Following the notation from \cite{2a}, we denote it by $\coprod_{m\geq m_0} R^{\Lambda,m}_S$. The superscript ``$\Lambda$" is because for the construction of the quot scheme of torsion-free Hitchin pairs one views them as $\Lambda$-modules, where $\Lambda=\Sym~ (\omega^{\vee}_{\mathcal X/S})$ \cite{32}. We denote the stack of families of torsion-free sheaves over $\mathcal X/S$ of rank $n$ and degree $d$ by $TF(\mathcal X/S)$. The stack $TF(\mathcal X/S)$ is a reduced and irreducible Artin stack.

\smallskip 

\subsubsection{\textbf{Classical Artin stack of Gieseker-Higgs bundles}} \label{class}

Given any derived Artin stack $F: \textsf{cdga} \longrightarrow \mathbb S$ one can define a classical/ordinary Artin stack by considering the composition functor $F^{cl}: \textsf{alg}_{\mathbb k}\longrightarrow \textsf{cdga}_{\mathbb k}\longrightarrow \mathbb S$. We call $F^{cl}$ as the underlying classical Artin stack of $F$. Following this notation, we denote the underlying classical Artin stack of the derived stack $M$ of Higgs bundles over the family of curves $\mathcal X_{\mathfrak M}/\mathfrak M$ by $M^{cl}$.
\
Let us denote by $\Coh(\mathcal X/S)$ the Artin stack of coherent $\mathcal O_{\mathcal X}-$ modules that are flat over $S$. There is a natural map
$$\theta: M^{cl}\longrightarrow \Coh(\mathcal X/S)$$
which is given by the push forward of the underlying bundle $\pi_*\mathcal E$, where $\pi: \mathfrak X\longrightarrow \mathcal X$ is the modification morphism.
\
Consider the open sub-stack $TF(\mathcal X/S)\subset \Coh(\mathcal X/S)$ consisting of torsion-free sheaves. The stack $\theta^{-1}(TF(\mathcal X/S))$ is an open substack of $M^{cl}$. Consider the natural map $\pi^*\pi_*\mathcal E\longrightarrow \mathcal E$ on the universal curve $\mathfrak X$ on $M^{cl}$. Consider the sub-stack of $M^{cl}$ where the map is surjective. It is again an open sub-stack. Let us denote it by $M^{gg}$ (the superscript "gg" stands for globally generated). Then we define an open sub-stack
$$M^{cl}_{Gie}:=\theta^{-1}(TF(\mathcal X/S))\cap M^{gg}$$
It is an open sub-stack consisting of Gieseker-Higgs bundles of rank $n$ and degree $d$. There is a natural map
\begin{align}\label{pushforward112}
M^{cl}_{Gie}\longrightarrow \TFH(\mathcal X/S)\\
(\mathfrak X, \mathcal E, \phi)\mapsto (\pi_*\mathcal E, \pi_*\phi)
\end{align}
where $\pi: \mathfrak X\longrightarrow \mathcal X$ is the modification map. We denote by $M^{^{cl}}_{Gie,0}$ the closed fibre of the map $M^{^{cl}}_{Gie}\longrightarrow S$.
\begin{Rem}
    We remind the reader that the stacks $M^{cl}$ and $M^{cl}_{Gie}$ are defined over the stack of expanded degenerations $\mathfrak M$. But the codomain stacks $\Coh(\mathcal X/S)$ and $\TFH(\mathcal X/S)$ live only on $S$.
\end{Rem}

\smallskip

\subsection{Construction of an atlas for $N_{Gie}$}\label{Atlas1}

Let us denote by $N_{Gie}$ the classical stack of Gieseker vector bundles. It can be seen as the closed substack of $M^{cl}_{Gie}$ consisting of the Gieseker-Higgs bundles $(\mathcal E, \phi)$, whose Higgs field $\phi=0$. In this subsection, we will give the construction of an atlas of the stack $N_{Gie}$. Let $\mathcal{O}_{\mathcal{X}/S}(1)$ be a relatively ample line bundle for the family of curves $\mathcal{X}/S$. As before, we fix rank $n$ and degree $d$.

\begin{Def}
Let $m, N(m)$ be positive integers with $N(m) \geq n$. Define the functor
\[
\mathcal{G}^{m}_S : {Sch}/S \longrightarrow {Sets}
\]
by
\[
\mathcal{G}^{m}_S(T) = \{(\Delta_T, V_T)\},
\]
where

\begin{enumerate}
\item $\Delta_T \subset \mathcal{X} \times_S T \times \mathrm{Grass}(N(m),n)$ is a closed subscheme,
\item the projection $j : \Delta_T \to T \times \mathrm{Grass}(N(m),n)$ is a closed immersion,
\item the projection $\Delta_T \to \mathcal{X} \times_S T$ is a modification,
\item the projection $p_T : \Delta_T \to T$ is a flat family of Gieseker curves,
\item Let $\mathcal{V}$ be the tautological quotient bundle of rank $n$ on $\mathrm{Grass}(N(m),n)$ and $\mathcal{V}_T$ its pullback to $T \times \mathrm{Grass}(N(m),n)$. Then
   \[
   V_T := j^*(\mathcal{V}_T)(-m)
   \]
   is a Gieseker vector bundle on the modification $\Delta_T$ of rank $n$ and degree $d$.
\item For each $t \in T$, the quotient $\mathcal{O}_{\Delta_t}^{N(m)} \twoheadrightarrow V_t(m)$ induces an isomorphism
   \[
   H^0(\Delta_t, \mathcal{O}_{\Delta_t}^{N(m)}) \;\cong\; H^0(\Delta_t, V_t(m))
   \]
   and $H^1(\Delta_t, V_t(m)) = 0$.
\end{enumerate}
\end{Def}

Let $P(m)$ be the Hilbert polynomial of the closed subscheme $\Delta_s$ of $\mathcal{X}_s \times \mathrm{Grass}(N(m),n)$ with respect to the polarisation $\mathcal{O}_{\mathcal{X}_s}(1) \boxtimes \mathcal{O}_{\mathrm{Grass}(N(m),n)}(1)$, where $\mathcal{O}_{\mathrm{Grass}(N(m),n)}(1) = \det \mathcal{V}$.

It is shown in \cite[Proposition 8]{25} that the functor $\mathcal{G}^{m}_S$ is represented by a $PGL(N(m))$-invariant open subscheme $\mathcal{Y}^{m}_S$ of the Hilbert scheme
\[
\mathcal{H}_S = \mathrm{Hilb}^{P(m)}(\mathcal{X} \times \mathrm{Grass}(N(m),n)).
\]

Finally, the disjoint union
\[
\mathcal{Y}_S := \coprod_{N(m),\, m \geq 0} \mathcal{Y}^{m}_S \longrightarrow N_{Gie}
\]
is an atlas. The varieties $\mathcal{Y}^{m}_S$ are smooth and the projection map to the stack is smooth.

\subsection{Construction of an atlas for $M^{cl}_{Gie}$}\label{at2}

\begin{Def}\label{Ydef}
Define the functor
\[
\mathcal{G}^{H,m}_S : {Sch}/\mathcal{Y}_S \longrightarrow {Groups}
\]
by
\[
\mathcal{G}^{H,m}_S(T) = H^0\bigl(T, (p_T)_*(\End \mathcal{V}_T \otimes \omega_{\Delta_T/T})\bigr),
\]
where $p_T : \Delta_T := \Delta_{\mathcal{Y}_S} \times_{\mathcal{Y}_S} T \to T$ is the projection, and $\omega_{\Delta_T/T}$ is the relative dualizing sheaf of the family of curves $p_T$.
\end{Def}

Since each $\mathcal{Y}^m_S$ is reduced, the functor $\mathcal{G}^{H,m}_S$ is representable: there exists a linear scheme $\mathcal{Y}^{H,m}_S$ over $\mathcal{Y}^m_S$ representing it.

For an $S$-scheme $T$, a point in $\mathcal{G}^{H,m}_S(T)$ consists of $(V_T, \phi_T)$ where $(V_T)$ is in $\mathcal{G}^m_S(T)$ and $(V_T, \phi_T)$ is a Gieseker–Higgs bundle.

The disjoint union
\[
\mathcal{Y}^H_S := \coprod_{N(m),\, m \geq 0} \mathcal{Y}^{H,m}_S \longrightarrow M^{cl}_{Gie}
\]
is an atlas. The projection map to the stack is smooth.

\begin{Rem}\label{equivar}
There is a natural morphism
\[
\tilde{\theta} : \mathcal{Y}^{H,m}_S \longrightarrow R^{\Lambda,m}_S
\]
given by pushforward
\[
(\Delta_T, V_T, \phi_T) \;\longmapsto\; \bigl( (\pi_T)_* V_T,\; (\pi_T)_* \phi_T \bigr),
\]
where $\pi_T : \Delta_T \to \mathcal{X} \times_S T$ is the modification map. This morphism is $GL_{N(m)}$-equivariant.
\end{Rem}

\medskip

\subsection{\textbf{Dimension and local properties of \texorpdfstring{$M^{cl}_{Gie,0}$}{MclGie0}}}
In this subsection, we will compute the dimension of $M^{cl}_{Gie,0}$ and show that it is a local complete intersection. 

\medskip

\subsubsection{\textbf{Relative log-symplectic reduction}}\label{Reduc} Fix one atlas component $\mathcal Y^m_S$ of $M^{cl}_{Gie}$. The group $GL_{N(m)}$ acts on $\mathcal Y^m_S$ and the quotient stack $[\mathcal Y^m_S/GL_{N(m)}]$ is an open sub-stack of $N_{Gie}$. 

Consider the relative log-cotangent bundle $\Omega^{\log}_{\mathcal Y^m_S/S}$ of the morphism $\mathcal Y^m_S\to S$. Since the $GL_{N(m)}$ action preserves the normal crossing divisor $\mathcal Y^m_0$ (the closed fibre of $\mathcal Y^m_S\longrightarrow S$), the action of $GL_{N(m)}$ lifts to an action on $\Omega^{\log}_{\mathcal Y^m_S/S}$ with a moment map $\mu_{\log}: \Omega^{\log}_{\mathcal Y^m_S/S}\longrightarrow \mathfrak{gl}_{N(m)}^*$. The action of $GL_{N(m)}$ has a generic stabiliser $\mathbb G_m$; therefore the map $\mu_{\log}$ actually factors through $\Omega^{\log}_{\mathcal Y^m_S/S}\longrightarrow \mathfrak{pgl}_{N(m)}^*$. Since we know that there exists an open subset of $\mathcal Y^m_S$ where the action of $GL_{N(m)}$ has stabiliser isomorphic to $\mathbb G_m$ (namely, the locus of stable vector bundles), therefore the map $\Omega^{log}_{\mathcal Y^m_S/S}\longrightarrow \mathfrak{pgl}_{N(m)}^*$ is surjective. Then one can show that 

\begin{equation}\label{qqq112}
    \mathcal Y^{H, m}_S=\mu_{\log}^{-1}(0).
\end{equation}

To see this, note that $\mathcal Y^m_S$ is a principal $GL_{N(m)}$ bundle over an open subset of the stack $N_{Gie}$. Let $a$ denote the projection $a: \mathcal Y^m_S\longrightarrow N_{Gie}$. Let us write the cotangent sequence of the sheaf of relative logarithmic differential forms. 
\begin{equation}
    0\longrightarrow a^*(T^{\log}_{N_{Gie}/S})^*\longrightarrow (T^{\log}_{\mathcal Y^m_S/S})^*\longrightarrow \mathfrak gl^*_{N(m)}\otimes \mathcal O_{\mathcal Y^m_S}
\end{equation}
where the last map is given by $\mu_{\log}$. Notice that $[(T^{\log}_{\mathcal Y^m_S/S})^*\longrightarrow \mathfrak gl^*_{N(m)}\otimes \mathcal O_{\mathcal Y^m_S}]$ is the cotangent complex of the stack $N_{Gie}$. Therefore, using standard arguments from the deformation theory of vector bundles, it follows that the $0$-th cohomology of the complex at a point $[\mathcal O^{N(m)}\rightarrow \mathcal E]\in \mathcal Y^m_S$ is isomorphic to $H^1(\End \mathcal E)^{\vee}\cong \Hom(\mathcal E, \mathcal E\otimes \omega)$. Therefore, by Definition \ref{Ydef}, $\mathcal Y^{H, m}_S=\mu_{\log}^{-1}(0)$. 

Therefore, $$[\mu_{\log}^{-1}(0)/GL_{N(m)}]$$ is an open subset of $M^{cl}_{Gie}$. 

We will compute the dimension of $\mu_{\log}^{-1}(0)$ and show that it is local complete intersection. We begin by recalling a result on the dimension of the image of a cotangent fibre under the moment map. 

\

\begin{Lem}\label{Moment}
$\dim \mu_{\log}(\Omega^{\log}_{\mathcal Y^m_S/S,y})\geq ~\dim~\big (\frac{\mathfrak{gl}_{N(m)}}{\mathfrak{gl}_{N(m), y}}\big )^*$, where $\mathfrak{gl}_{N(m),y}$ is the Lie algebra of the stabilizer of $y \in \mathcal Y^m_S$ under the action of $GL_N(m)$.
\end{Lem}

\begin{proof}
Consider the diagram
\begin{equation}
\begin{tikzcd}
0\arrow{r} & \mathfrak{gl}_{N(m), y}\arrow{r} & \mathfrak{gl}_{N(m)}\arrow{r}\arrow{d} & \frac{\mathfrak{gl}_{N(m)}}{\mathfrak{gl}_{N(m), y}}\arrow{r} & 0\\
0\arrow{r} & K\arrow{r} & T^{log}_{\mathcal Y^m_S/S,y}\arrow{r} & T_{\mathcal Y^m_S/S,y}
\end{tikzcd}
\end{equation}

Here $K$ denotes the kernel of the natural map $T^{\log}_{\mathcal Y^m_S/S,y}\longrightarrow  T_{\mathcal Y^m_S/S,y}$. Notice that the map $\mathfrak{gl}_{N(m)}\longrightarrow T_{\mathcal Y^m_S/S,y}$ is the differential of the orbit map, and it factors through $T^{log}_{\mathcal Y^m_S/S,y}$ because the action of $GL_{N(m)}$ preserves the normal crossing divisor. It is well known that the rightmost vertical map is injective \cite[Lemma 2.4.1]{32e}. Therefore, we can complete the diagram. 
\begin{equation}
\begin{tikzcd}
0\arrow{r} & \mathfrak{gl}_{N(m), y}\arrow{r}\arrow{d} & \mathfrak{gl}_{N(m)}\arrow{r}\arrow["\iota"]{dr}\arrow["\iota_{log}"]{d} & \frac{\mathfrak{gl}_{N(m)}}{\mathfrak{gl}_{N(m), y}}\arrow{r}\arrow[hook]{d} & 0\\
0\arrow{r} & K\arrow{r} & T^{\log}_{\mathcal Y^m_S/S,y}\arrow{r} & T_{\mathcal Y^m_S/S,y}
\end{tikzcd}
\end{equation}
Therefore we see that $\ker (\iota_{\log})\subset \ker (\iota)=\mathfrak{gl}_{N(m), y}$. Since moment and log-moment maps are the dual of the maps $\iota$ and $\iota_{\log}$, therefore $\dim~ \mu_{\log}(\Omega^{\log}_{\mathcal Y^m_S/S,y})\geq ~\dim~\big (\frac{\mathfrak{gl}_{N(m)}}{\mathfrak{gl}_{N(m), y}}\big )^*$. 
\end{proof}

\

\begin{Prop}\label{new1}
The closed fibre $N^{cl}_{Gie,0}$ is an irreducible, equidimensional, almost very good stack (\cite[Definition 2.1.2]{32e}) with normal crossing singularities.
\end{Prop}

\begin{proof}
From \cite[Theorem 9.5]{16a},it follows that the normalization of $N^{cl}_{Gie,0}$ is a bundle $KGL_n$ on the stack of vector bundles $\Bun(\tilde{X}_0)$ of rank $n$ and degree $d$ on the normalisation $\tilde{X}_0$ of the nodal curve $X_0$. Here, $KGL_n$ denotes the compactification of $GL_n$ constructed by Kausz \cite{16b}. More precisely, let $E$ be the universal $GL_n$ bundle over $\tilde{X_0}\times \Bun(\tilde{X_0})$. Consider the $GL_n\times GL_n$ bundle $E_{x_1}\times_{\Bun(\tilde{X_0})} E_{x_2}$ over $\Bun(\tilde{X_0})$. Then the associated $KGL_n$ fibration $(E_{x_1}\times_{Bun(\tilde{X_0})} E_{x_2})\times_{GL_n\times GL_n} KGL_n\cong \widetilde{N^{cl}_{Gie,0}}$. This is the stack of Gieseker vector bundle (\cite[Definition 4.7]{16a}), which is equivalent to (\cite[Lemma 5.6]{DI}) the stack of marked Gieseker vector bundles (that is, a Gieseker vector bundle with a marked node (\cite[Definition 5.1]{DI})). It is obvious that the automorphism group of a marked Gieseker vector bundle is isomorphic to the automorphism group of the corresponding Gieseker vector bundle. Therefore, it is enough to show that the stack of Gieseker vector bundle data is smooth, equidimensional, irreducible, and almost very good. 
\vspace{.3cm}

Now let us recall that the map $\tilde{\pi}: \widetilde{N^{cl}_{Gie,0}}\longrightarrow \Bun(\tilde{X}_0)$ is given by $(X^{m,n}, s_1, s_2, \tilde{\mathcal E}, \phi)\mapsto h_*(\mathcal E(-s_1-s_2))(p_1+p_2)$, where $h: X^{m,n}\longrightarrow \tilde{X}_0$ is the modification map (\cite[Lemma 9.3]{16a}). Here $(X^{m,n}, s_1, s_2, \tilde{\mathcal E}, \phi)$ is a Gieseker vector bundle data and here $\phi$ is not a Higgs bundle but an identification between the fibres $\tilde{\mathcal E}_{s_1}\xrightarrow{\cong} \tilde{\mathcal E}_{s_2} $ (\cite[Definition 4.7]{16a}). It is straightforward to check that the induced map \[\Aut(X^{m,n}, s_1, s_2, \tilde{\mathcal E}, \phi)\longrightarrow \Aut(h_*(\mathcal E(-s_1-s_2))(p_1+p_2))\] is injective. For notational
convenience, denote $Y:= \widetilde{N^{cl}_{Gie,0}}$ and $Z:=\Bun(\tilde{X}_0)$. 

From \cite[Proposition 2.1.2]{3a} and \cite[Definition 2.1.2, Remark 2.1.3]{32e}, it follows that $Z$ is smooth, irreducible, almost very good stack. More precisely, the $\codim_Z(Z_k)>k~~\forall k\geq 1$, where $Z_k:=\{z\in Z~|~\dim~\Aut~(z)=1
+k \}$. From the fact that $\Aut(X^{m,n}, s_1, s_2, \tilde{\mathcal E}, \phi)\subset \Aut(h_*(\mathcal E(-s_1-s_2))(p_1+p_2))$, it follows that $\tilde{\pi}(Y_k)\subset Z_k$. Therefore, $Y_k\subset \tilde{\pi}^{-1}(Z_k)$ and \[\dim~Y_k\leq \dim~~\tilde{\pi}^{-1}(Z_k)=\dim ~~Z_k+\dim~~KGL_n.\] Now for all $k>0$ we have \begin{equation*}\begin{split}\codim_Y(Y_k)= & \dim~~ Y-\dim~~ Y_k \\ & \geq ~~\dim~~ Y-\dim~~ Z_k-\dim~~ KGL_n=\dim~~ Z-\dim~~ Z_k \\ & = \codim_Z(Z_k)>k.
\end{split}
\end{equation*}
Therefore, $\widetilde{N^{cl}_{Gie,0}}$ is an irreducible smooth stack, almost very good and $N^{cl}_{Gie,0}$ is an irreducible, almost very good stack. 
\end{proof}

\

\begin{Thm}\label{dim111}
The stack $M^{cl}_{Gie,0}$ is an irreducible local complete intersection of pure dimension $2\dim N_{Gie,0} + 1$.
\end{Thm}

\begin{proof}
It is enough to show that $\mu_{\log,0}^{-1}(0)=\mathcal Y^{H,m}_0$ is a local complete intersection of dimension $2\cdot \dim~~N_{Gie,0}+1+\dim~GL_{N(m)}$. Here, $\mu_{\log,0}$
denotes the restriction of $\mu_{\log}$ to the closed fiber of $\Omega^{\log}_{\mathcal Y^m_S/S}\longrightarrow S$. Since the map $\mu_{\log,0}: \Omega^{\log}_{\mathcal Y^m_0}\longrightarrow \mathfrak{pgl}^*_{N(m)}$ is surjective, the dimension of the generic fiber is equal to $\dim~\mathcal Y^m_0-\dim~\mathfrak{gl}^*_{N(M)}+1$. Therefore, for every irreducible component $I$ of $\mu^{-1}_{\log,0}(0)$, we have 
\begin{equation}\label{geq}
\dim~I\geq \dim~\mathcal Y^m_0-\dim~\mathfrak{gl}^*_{N(M)}+1.
\end{equation}
Suppose that there exists a component $I\subseteq \mu^{-1}_{\log,0}(0)$ that does not dominate $N_{Gie,0}$. In that case $(q\circ p)(I)\subseteq N_k$ for some $k\geq 1$, where $N_k:=\{x\in N:=N^{cl}_{Gie, 0} | \dim~Stab(x)=k+1\}$ and $p$ is the restriction to $\mu^{-1}_{\log, 0}(0)$ of the projection map $\pi: \Omega^{\log}_{\mathcal Y^m_S/S}\rightarrow \mathcal Y^m_0$. 

\begin{equation}
\begin{tikzcd}
   I\subseteq \mu^{-1}_{\log, 0}(0)\arrow["p"]{drr} \arrow[hook]{rr} && \Omega^{log}_{\mathcal Y^m_0}\arrow["\pi"]{d}\arrow{dr}\arrow["\mu_{\log, 0}"]{r} & \mathfrak{pgl}^*_{N(m)} \\
    && \mathcal Y^m_0 \arrow["q"]{r}  & N:=N^{cl}_{Gie, 0}
\end{tikzcd}
\end{equation}

A priori, the generic point of $(q\circ p)(I)$ may belong to the singular locus of $N$. But, in any case, we have 
\begin{align*}
 \dim~I &\leq \underbrace{\dim~p(I)}_{\text{dim of the base}}+\underbrace{\dim~\mathcal Y^m_0-\dim~\big (\frac{\mathfrak{gl}_{N(m)}}{\mathfrak{gl}_{N(m),y}}\big )^*}_{\text{upper bound of the dim of the fibre at the generic point of p(I)}}
 ~~~~(\text{using Lemma}~~\ref{Moment})\\
&< (\dim~\mathcal Y^m_0-k)+\dim~\mathcal Y^m_0-\dim~\big (\frac{\mathfrak{gl}_{N(m)}}{\mathfrak{gl}_{N(m),y}}\big )^*\hspace{10em}\\
&(\text{because}~ N_{Gie,0}~\text{ is a almost very good stack})\\
&=2\cdot \dim~\mathcal Y^m_0-\dim~\mathfrak{gl}^*_{N(m)}+(\dim~\mathfrak{gl}^*_{N(m),y}-(k+1))+1\hspace{5em}\\
&=2\cdot \dim~\mathcal Y^m_0-\dim~\mathfrak{gl}^*_{N(m)}+1\hspace{16em}\\
&(\text{because} ~~\dim~\mathfrak{gl}^*_{N(m),y}=(k+1), \text{by definition})\\
&\leq \dim~~I  ~~~~(\text{from}~~\eqref{geq})
\end{align*}
where $y$ denotes the generic point of $p(I)$. But this is a contradiction. Therefore, $k$ must be equal to $0$, which implies that every component $I$ dominates $\mathcal Y^m_0$ and $\mu^{-1}_{\log,0}(0)$ is an equi-dimensional of dimension $2\cdot \dim~\mathcal Y^m_0-\dim~\mathfrak{gl}^*_{N(m)}+1$. Once we know the dimension, it is easy to see that $\mu^{-1}_{log,0}(0)$ is a local complete intersection. Therefore, $M^{cl}_{Gie,0}$ is an equidimensional local complete intersection stack. 

We define \[M^{cl,sm}_{Gie,0}:=\{(\mathcal E, \phi)\in M^{cl}_{Gie,0} \ \ | \ \  \mathbb H^2(\mathcal C(\mathcal E, \phi))\cong \mathbb k\},\] where $\mathcal C(\mathcal E, \phi))$ denotes the complex $[\End \mathcal E\xrightarrow{[-, \phi]} \End \mathcal E\otimes \omega ]$. This substack is obviously an open subset of $M^{cl}_{Gie,0}$ because the minimum dimension of $\mathbb H^2(\mathcal C(\mathcal E, \phi))$ for any Gieseker-Higgs bundle is $1$. Using the fact that $N^{cl}_{Gie,0}$ is an irreducible, almost very good stack with a generic stabilizer of dimension $1$ and the fact that $M^{cl}_{Gie,0}$ is equi-dimensional, it is not difficult to see that $M^{cl, sm}_{Gie,0}$ is also dense in $M^{cl}_{Gie,0}$. 

But note that $M^{cl, sm}_{Gie,0}$ is a vector bundle over $N_{Gie, 0}$ with fibers $H^0(\End~\mathcal E\otimes \omega)$ (which remains constant in this locus). Since $N_{Gie, 0}$ is irreducible, $M^{cl,sm}_{Gie,0}$ is irreducible. Therefore, $M^{cl}_{Gie,0}$ is also irreducible because $M^{cl,sm}_{Gie,0}$ is dense in $M^{cl}_{Gie,0}$. 
\end{proof}

\begin{Cor} \label{LCI}
The classical Artin stack $M^{cl}_{Gie}$ is a local complete intersection of pure dimension $2\cdot \dim N_{Gie, 0}+1$.
\end{Cor}

\begin{proof}
Since the geometric fibres of the morphism $M^{cl}_{Gie} \to S$ all have (pure) dimension $2n^2(g-1)+1$, it follows that
\[
\dim M^{cl}_{Gie} = 2n^2(g-1) + 1 + \dim S = 2n^2(g-1) + 2.
\]

As explained in:

-[Subsection \ref{Atlas1}]\[
\mathcal{Y}_S := \coprod_{N(m),\, m \geq 0} \mathcal{Y}^{m}_S \longrightarrow N_{Gie}
\]
is an atlas. 

-[Subsubsection \ref{Reduc}] $M^{cl}_{Gie}$ admits an open cover by quotient stacks of the form 

$$[\mathcal Y^{H, m}_S/GL_{N(m)}]=[\mu_{\log}^{-1}(0)/GL_{N(m)}],$$ 

where $\mu_{\log} : \Omega^{\log}_{\mathcal{Y}^m_S/S} \to \mathfrak{pgl}_{N(m)}^*$ is the moment map associated to the $GL_{N(m)}$-action on a given atlas component $\mathcal{Y}^m_S$ of $N_{Gie}$.

-[Subsection \ref{at2}] The disjoint union
\[
\mathcal{Y}^H_S := \coprod_{N(m),\, m \geq 0} \mathcal{Y}^{H,m}_S \longrightarrow M^{cl}_{Gie}
\]
is an atlas. 

The dimension of the source $\Omega^{\log}_{\mathcal{Y}^m_S/S}$ of the map $\mu_{\log} : \Omega^{\log}_{\mathcal{Y}^m_S/S} \to \mathfrak{pgl}_{N(m)}^*$ is $2n^2(g-1) + N(m)^2 + \dim S$. Thus, the expected dimension of the zero locus $\mu_{\log}^{-1}(0)$ is $2n^2(g-1) + N(m)^2 + 1 + \dim S$. But we have already established that $\dim M^{cl}_{Gie} = 2n^2(g-1) + 2$, so the actual dimension of $\mu_{\log}^{-1}(0)$ is equal to its expected dimension $\dim \mu_{\log}^{-1}(0) = 2n^2(g-1) + N(m)^2 + 1 + \dim S$. Since $\Omega^{\log}_{\mathcal{Y}^m_S/S}$ is smooth, the zero locus $\mu_{\log}^{-1}(0)$ is therefore a local complete intersection. It follows that $M^{cl}_{Gie}$ is itself a local complete intersection.
\end{proof}

\

\nocite{*}
\bibliography{biblio}
\bibliographystyle{amsalpha}

\noindent
Oren Ben-Bassat \\
Department of Mathematics \\
University of Haifa \\
Mount Carmel \\
Haifa, 3498838, ISRAEL\\
ben-bassat@math.haifa.ac.il\\

\noindent
Sourav Das \\
Department of Mathematics\\
Indian Institute of Science Education and Research Tirupati\\
Srinivasapuram, Venkatagiri Road, Jangalapalli Village, Panguru\\ (G.P), Yerpedu Mandal, Tirupati District, Andhra Pradesh\\
517619, INDIA\\
sdas6565@gmail.com\\ 

\noindent
Tony Pantev \\
Department of Mathematics \\
209 S. 33rd Street \\
Philadelphia, PA, 19104, USA \\
tpantev@math.upenn.edu

\end{document}